\documentclass[a4paper, 11pt]{article}

\usepackage{amsmath, amsthm,amsfonts, graphicx, color, enumitem, url, hyperref}
\usepackage[justification=centering, font={sl,small}, labelfont={bf}, margin=1cm]{caption}
\usepackage[font={sl,footnotesize}]{subcaption}

\hypersetup{colorlinks=true, citecolor=darkblue, linkcolor=darkblue}

\usepackage[margin=2cm, marginparwidth=2cm,marginparsep=0cm]{geometry}

\makeatletter
\def\namedlabel#1#2{\begingroup
   \def\@currentlabel{#2}%
   \label{#1}\endgroup
}
\makeatother

\makeatletter
\renewcommand{\pod}[1]{\mathchoice
  {\allowbreak \if@display \mkern 6mu\else \mkern 6mu\fi (#1)}
  {\allowbreak \if@display \mkern 6mu\else \mkern 6mu\fi (#1)}
  {\mkern4mu(#1)}
  {\mkern4mu(#1)}
}

\newtheorem{theorem}{Theorem}[section]
\newtheorem{proposition}[theorem]{Proposition}
\newtheorem{corollary}[theorem]{Corollary}
\newtheorem{lemma}[theorem]{Lemma}

\theoremstyle{definition}
\newtheorem{definition}[theorem]{Definition}
\newtheorem{example}[theorem]{Example}

\newtheorem{remark}[theorem]{Remark}
\newtheorem{question}[theorem]{Question}

\makeatletter
\newtheorem*{rep@theorem}{\rep@title}
\newcommand{\newreptheorem}[2]{%
\newenvironment{rep#1}[1]{%
 \def\rep@title{#2 \ref{##1}}%
 \begin{rep@theorem}}%
 {\end{rep@theorem}}}
\makeatother

\newreptheorem{theorem}{Theorem}

\newcommand{\productp}[2]{\Delta_{{#1}}\times\Delta_{{#2}}} 
\newcommand{\product}{\productp{m-1}{n-1}} 
\newcommand{\productn}{\productp{n-1}{n-1}} 
\newcommand{\kskel}[2]{\Delta_{#1}^{({#2})}}
\newcommand{\semiskelp}[3]{\Delta_{{#1}}^{({#3})}\times\Delta_{{#2}}} 
\newcommand{\semiskel}{\semiskelp{m-1}{n-1}{k-1}} 

\newcommand{\bipartite}{\bipartitep{m}{n}} 
\newcommand{\bipartitep}[2]{K_{{#1},{\overline #2}}} 

\newcommand{\gr}[2]{\mathcal{G}_{{#1}\times {#2}}} 
\newcommand{\grid}{\gr{m}{n}}

\newcommand{\cM}{\mathcal{M}} 
\newcommand{\cT}{\mathcal{T}} 

\newcommand{\Match}{\mathfrak{m}} 
\newcommand{\Mext}{\cM^\text{ext}} 

\newcommand{\set}[2]{\left\{{#1}\, :\, {#2}\right\}} 

\newcommand{\semibdy}[2]{\partial\left(\Delta_{{#1}}\right)\times \Delta_{{#2}} } 

\newcommand{\bfC}{\mathbf{C}} 
 
\newcommand{\dyckptri}[1]{\boldsymbol{\mathfrak{D}}_{#1}}
\newcommand{\ratdyckptri}[2]{\boldsymbol{\mathfrak{D}}_{(#1,#2)}}
\newcommand{\ratdyckptriext}[2]{\ratdyckptri{#1}{#2}^\text{ext} }

\newcommand{\dyckptriext}[1]{\dyckptri{#1}^\text{ext} }
\newcommand{\bdydyckptriext}[1]{\dyckptri{#1}^{\partial\text{ext}} }
\newcommand{\flippedD}{\reflectbox{\ensuremath{\boldsymbol{\mathfrak{D}}}}}
\newcommand{\dyckptriflip}[1]{\flippedD_{#1}} 
\newcommand{\bdydyckptriflip}[1]{\dyckptriflip{#1}^{\partial\text{ext}} } 

\newcommand{\catnum}[1]{{C}_{#1}} 
\newcommand{\ratcatnum}[2]{{C}{(#1,#2)}} 

\newcommand{\conv}{\mathrm{conv}} 

\newcommand{\bfe}{\mathbf{e}}
\newcommand{\bfv}{\mathbf{v}}

\newcommand{\bfF}{\mathbf{F}}

\newcommand{\NN}{\mathbb{N}}
\newcommand{\RR}{\mathbb{R}}
\newcommand{\TT}{\mathbb{T}}

\newcommand{\ol}[1]{\overline{#1}}

\definecolor{darkblue}{rgb}{0,0,0.7} 
\newcommand{\darkblue}{\color{darkblue}} 
\newcommand{\defn}[1]{\emph{\darkblue #1}} 

\graphicspath{{figures/}}

\AtEndDocument{\bigskip{\footnotesize%
  \textsc{Department of Mathematics and Statistics, York University, Toronto, Ontario M3J 1P3, Canada} \par  
  \textit{E-mail address}: \texttt{ceballos@mathstat.yorku.ca} \par
  \textit{URL}: \url{http://garsia.math.yorku.ca/~ceballos/}\\[0.3cm]
    \textsc{Institut f\"ur Mathematik, Freie Universit\"at Berlin, 14195 Berlin, Germany} \par  
  \textit{E-mail address}: \texttt{arnaupadrol@zedat.fu-berlin.de} \par
  \textit{URL}: \url{http://page.mi.fu-berlin.de/arnaupadrol/} \\[0.3cm]
  \textsc{Institut f\"ur Algebra und Geometrie, Otto-von-Guericke-Universit\"at Magdeburg, 39016 Magdeburg, Germany} \par  
  \textit{E-mail address}: \texttt{camilo.sarmiento@ovgu.de} \par
  \textit{URL}: \url{http://www-e.uni-magdeburg.de/sarmient/} \par
}}

\begin{document}

\title{Dyck path triangulations and extendability}

\author{Cesar Ceballos\thanks{Supported by the government of Canada through an
NSERC Banting Postdoctoral Fellowship. He was also supported by a York
University research grant.} \and Arnau Padrol\thanks{Supported by the DFG
Collaborative Research Center SFB/TR~109 ``Discretization
in Geometry and Dynamics''.}\and Camilo Sarmiento\thanks{Partially supported by
the IMPRS of the Max Planck Institute for Mathematics in the Sciences, Leipzig,
and by the CDS, Magdeburg.}
}

\maketitle

\abstract{
We introduce the Dyck path triangulation of the cartesian product of two simplices $\productn$. The maximal simplices of this triangulation are given by Dyck paths, and its construction naturally generalizes to produce triangulations of $\productp{rn-1}{n-1}$ using rational Dyck paths. Our study of the Dyck path triangulation is motivated by extendability problems of partial triangulations of products of two simplices. We show that whenever $m\geq k>n$, any triangulation of $\semiskel$ extends to a unique triangulation of $\product$. Moreover, with an explicit construction, we prove that the bound $k>n$ is optimal. We also exhibit interesting interpretations of our results in the language of tropical oriented matroids, which are analogous to classical results in oriented matroid theory. 
}


\section{Introduction}

The cartesian product of a standard $(m-1)$-simplex with a standard
$(n-1)$-simplex is the $(m+n-2)$-dimensional polytope
\[
\product:=\conv \{(\bfe_i,\bfe_j)\colon \bfe_i\in \Delta_{m-1},\
\bfe_j\in\Delta_{n-1}\}\subset\RR^{m+n},
\]
where $\bfe_i$ and $\bfe_j$ range over the standard basis vectors of
$\RR^m$ and $\RR^n$, respectively. 

Triangulations of the product of two simplices are intricate objects
that have been extensively studied with various purposes.
They are a key ingredient for understanding
triangulations of products of polytopes~\cite{DeLoera1996,Haiman1991,
OrdenSantos2003,Santos2000}. Via the Cayley trick, they are in bijection with fine mixed subdivisions of a dilated simplex $m\Delta_{n-1}$~\cite{Santos2005}, which provides a relation to tropical (pseudo) hyperplane arrangements and tropical oriented 
matroids~\cite{ArdilaDevelin2009,SturmfelsDevelin2004}. Moreover, they have also attracted interest in algebraic
geometry and commutative algebra~\cite{BabsonBillera,ConcaHostenThomas2005,GKZ,Sturmfels1996} and in Schubert calculus~\cite{ArdilaBilley2007}.

In this paper, we present an intriguing family of triangulations of~$\productn$
that we call \defn{Dyck path triangulations}, whose maximal simplices are described in
terms of Dyck paths in a $n \times n$ grid under a cyclic action. The maximal simplices of the Dyck path triangulation
of~$\productp{2}{2}$ and the corresponding fine mixed subdivision of~$3\Delta_2$
are illustrated in Figure~\ref{fig:dyck3_introduction}.
 
\begin{figure}[h]
\centering
 \includegraphics[width=0.45\textwidth]{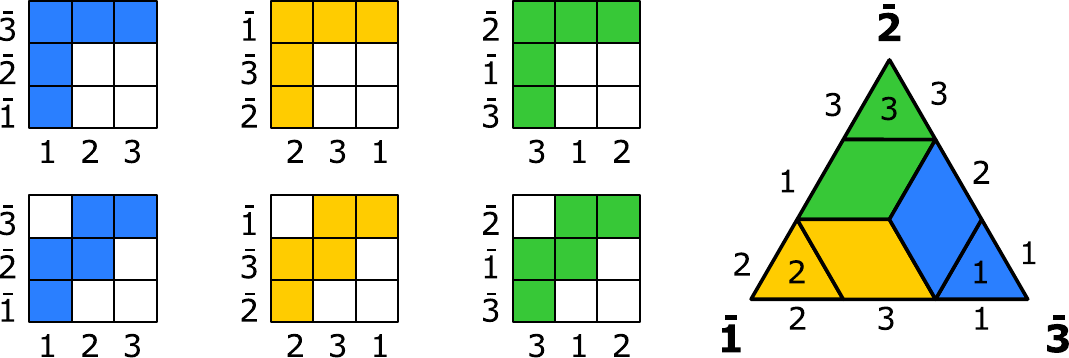}  
 \caption{The Dyck path triangulation of $\productp{2}{2}$ drawn as a subdivision of 3$\Delta_2$.}
 \label{fig:dyck3_introduction} 
 \end{figure}

Besides the combinatorial beauty of these triangulations, they are motivated by extendability problems of partial triangulations of $\product$.
The $(k-1)$-skeleton of $\Delta_{m-1}$, which we denote by~$\kskel{m-1}{k-1}$, is the polyhedral complex of all faces of $\Delta_{m-1}$ of dimension less than or equal to $k-1$.
A \defn{partial triangulation} of $\product$ is a triangulation of the polyhedral complex~$\semiskel$. Such a triangulation is said to be \defn{extendable} if it is equal to the restriction of a
triangulation of~$\product$ to~$\semiskel$. The smallest example of a non-extendable partial triangulation is shown in Figure~\ref{fig:nonextendable_a}; a more interesting example due to Santos~\cite{Santos2013} is shown in Figure~\ref{fig:nonextendable_b}. 

\begin{figure}[htbp]
 \centering
 \begin{subfigure}{0.45\textwidth} \centering
 \includegraphics[width=0.32\textwidth]{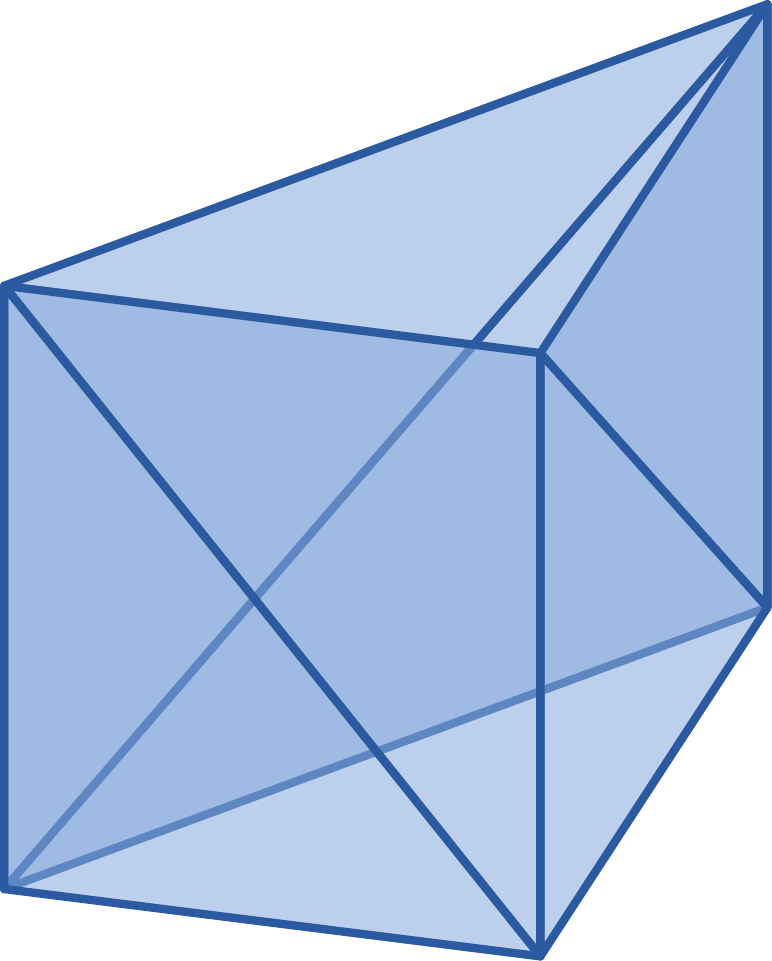}  
 \caption{A non-extendable triangulation of~$\semiskelp{2}{1}{1}$.}
 \label{fig:nonextendable_a} 
 \end{subfigure} 
 \begin{subfigure}{0.54\textwidth} \centering
 \includegraphics[width=1\textwidth]{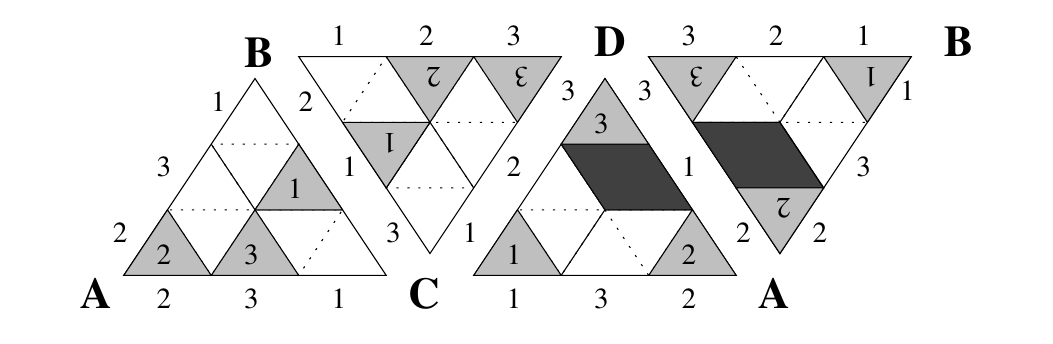}  
 \caption{A non-extendable triangulation of~$\semiskelp{3}{2}{2}$ (Santos).}
 \label{fig:nonextendable_b} 
 \end{subfigure} 
 \caption{Two examples of non-extendable partial triangulations. The triangulation in part (b), due to Francisco Santos~\cite{Santos2013}, is shown as a subdivision of $3\Delta_{3}^{(2)}= \partial \left( 3\Delta_{3} \right)$.}
 \end{figure}

The question of extendability of triangulations of $\semiskel$ was first systematically considered for $k=2$ by Ardila and Ceballos in~\cite{ArdilaCeballos2013}, who completely characterized the extendable triangulations of~$\semiskelp{2}{n-1}{1}$. There, in an attempt to prove the \emph{Spread Out Simplices Conjecture} of Ardila and Billey~\cite[Conjecture 7.1]{ArdilaBilley2007}, the authors formulated the \emph{Acyclic System Conjecture}~\cite[Conjecture 5.7]{ArdilaCeballos2013}, which concerned a sufficient condition for the extendability of triangulations of $\semiskelp{m-1}{n-1}{1}$. Shortly after, however, the Acyclic System Conjecture was disproved by Santos~\cite{Santos2013}. These results motivate the search for necessary and sufficient conditions for extendability.

Our first contribution is the following extendability theorem.

\begin{reptheorem}{theorem:extend}
Let $m,n,k$ be positive integers such that $m\geq k>n$. Every triangulation of~$\semiskel$ extends to a unique triangulation of $\product$.
\end{reptheorem}

In considering whether the bound $k>n$ in Theorem~\ref{theorem:extend} is optimal, we are led to the Dyck path triangulation of~$\productp{n-1}{n-1}$. This triangulation is our main tool to explicitly construct a family of partial triangulations that shows that the assertion of Theorem~\ref{theorem:extend} does not generally hold when~$m>k=n$.
\begin{reptheorem}{theorem:non-extend}
For every $n\geq 2$ there is a non-extendable triangulation of~$\semibdy{n}{n-1}$.
\end{reptheorem}

As suggested by its name, the Dyck path triangulation is based on Dyck paths and
is related to Catalan combinatorics. We devote the rest of the paper to the
study of this triangulation and its relatives. In particular, we present a
natural generalization in terms of \emph{rational Dyck paths} in the ``Fuss-Catalan" case $(rn,n)$. 
It would be interesting to know if it can be further generalized to other families
of rational Dyck paths.

Via the Cayley trick~\cite{Santos2005}, Theorems~\ref{theorem:extend}~and~\ref{theorem:non-extend} transform into statements about the extendability of ``partial fine mixed subdivisions'' of $n\Delta_{m-1}^{(k-1)}$, coming from triangulations of $\semiskel$, which we refer to as \defn{authentic subdivisions}.

\begin{corollary}
For $m\geq k>n$, every authentic subdivision of $n\Delta_{m-1}^{(k-1)}$ can be extended to a unique fine mixed subdivision of $n\Delta_{m-1}$.
Moreover, for every $n\geq 2$ there is a non-extendable authentic subdivision of $\partial \left( n\Delta_n \right)$.
\end{corollary}

Apart from providing a characterization of extendable of triangulations of $\semiskel$, our results admit additional interpretations that render them of broader interest.

On the one hand, Theorems~\ref{theorem:extend} and~\ref{theorem:non-extend} naturally translate into the language of tropical oriented matroids
(which we abbreviate as TOMs). This
concept was introduced by Ardila and Develin as an analogue of classical
oriented matroids for the tropical semiring~\cite{ArdilaDevelin2009}. The
combinatorics of an arrangement of $m$ tropical pseudohyperplanes in the
tropical space $\TT^{n-1}$ is captured by its TOM. 
The \emph{Topological Representation Theorem} establishes a correspondence
between TOMs (with parameters $(m,n)$) and subdivisions of
$\product$~\cite{ArdilaDevelin2009,Horn2012,OhYoo2011}. More concretely, triangulations of~$\product$ correspond
to generic TOMs and triangulations of $\semiskel$ correspond
to compatible collections of generic subarrangements of $k$ pseudohyperplanes; in this context, Lemma~\ref{lem:semiskel-unicity} reads as follows.

\begin{corollary}\label{cor:semiskel-unicity}
The TOM of any generic arrangement of tropical
pseudohyperplanes in $\TT^{n-1}$ is completely determined by the TOMs of
its subarrangements of $n$ pseudohyperplanes.
\end{corollary}

If $\cM$ is a TOM of an arrangement whose pseudohyperplanes have labels in 
$[m]$, denote by $\cM|_S$ the TOM of the subarrangement corresponding to the
hyperplanes with labels in $S\subseteq [m]$.
This turns
Theorems~\ref{theorem:extend}
and~\ref{theorem:non-extend} into the following statements,
respectively.

\begin{corollary}\label{cor:extend}
For each $S\in \binom{[m]}{n+1}$, let $\cM_S$ be the TOM of a generic
arrangement of $n+1$
pseudohyperplanes in $\TT^{n-1}$ with labels in $S$. If the matroids in this
collection are compatible, i.e. $\cM_S|_{S\cap T}=\cM_T|_{S\cap T}$ for every
$T,S\in \binom{[m]}{n+1}$, then there exists a unique arrangement of $m$
pseudohyperplanes in $\TT^{n-1}$ whose TOM $\cM$ fulfills $\cM|_S=\cM_S$.
\end{corollary}

\begin{corollary}\label{cor:non-extend}
There exists a collection of pairwise compatible TOMs on the subsets
in $\binom{[n+1]}{n}$ that cannot be completed to the TOM of an arrangement
of $n+1$
pseudohyperplanes in $\TT^{n-1}$.
\end{corollary}

These corollaries should be compared with analogue results in classical oriented
matroid
theory: every oriented matroid of rank~$n-1$ is completely determined by its
submatroids with $n$ elements  and every compatible
collection of submatroids with $n+1$ elements can be completed to a full
oriented matroid (cf.~\cite[Corollaries 3.6.3 and 3.6.4]{OrientedMatroids1993}).

On the other hand, Theorem~\ref{theorem:extend} can be regarded as a ``finiteness'' result for triangulations of $\product$: it says that, as long as $m\geq n+1$, triangulations of $\product$ are ``built'' from the collection of triangulations of $\productp{n}{n-1}$, no matter how large $m$ is. From this viewpoint, Theorem~\ref{theorem:extend} should be contrasted with recent results in commutative algebra regarding finiteness properties of the generating sets of certain families of polynomial ideals (see, for instance,~\cite{HillarSullivant12, HostenSullivant07, Snowden13}).

Here is the layout of the paper. The next section
contains some preliminaries concerning notation and
representations for triangulations of products of simplices. Section~\ref{sec:extend} contains the
proof of Theorem~\ref{theorem:extend}. The Dyck path triangulation is then
presented in Section~\ref{sec:DyckTriangulation}, along with the explicit
construction behind Theorem~\ref{theorem:non-extend}. The proof of the fact that
the Dyck path triangulation is indeed a triangulation is postponed to
Section~\ref{sec:matchings}, which also includes other revelant proofs;
in Section~\ref{sec:regularity} we prove that
the Dyck path triangulation is regular. Finally,
Section~\ref{sec:generalization} is devoted to generalizations of the
Dyck path triangulation.

\subsection*{Acknowledgements}

The authors want to thank Francisco Santos for many interesting and
enlightening conversations.


\section{Preliminaries}
In this section, in order to set our notation and
conventions, we recall some well-known facts related to triangulations of
$\product$ (we refer to~\cite[Section 6.2]{DeLoeraRambauSantos} for a more
detailed exposition).

\subsection*{Bipartite graph representation}

Let $\bipartite$ be the complete bipartite graph on $m+n$ vertices, whose parts
we label by $[m]$ and $[\overline n]$\footnote{Throughout we use overlined numbers and variables to
distinguish the vertices of both factors}. A vertex $(\bfe_i,\bfe_j)$ of $\product$
can be represented as the undirected edge $(i,\overline j)$ of $\bipartite$. It
turns out that independent sets, spanning sets and circuits of $\product$ are
easy to read from the \defn{bipartite graph representation}.

\begin{lemma}[{\cite[Lemma 6.2.8]{DeLoeraRambauSantos}}]\label{lem:bipartite}
 In the bipartite graph representation:
\begin{enumerate}
 \item A subset of vertices of $\product$ is affinely independent if and only if
the corresponding subgraph has no cycles
 . In particular, affine bases correspond to spanning trees.
 \item A subset of vertices of $\product$ is affinely spanning if and only if
the corresponding subgraph is connected and spanning.
 \item\label{it:bipartitecircuits} A subset of vertices of $\product$ is a
circuit if and only if the corresponding subgraph is a cycle. The positive and
negative elements of the circuit alternate along the cycle. In other words, two
edges of the cycle have the same sign (as elements in the circuit) if and only
if they are an even number of steps away from each other.
 In the standard realization of $\product$, all the coefficients of the affine
dependences corresponding to these circuits are $\pm 1$.
 \end{enumerate}
\end{lemma}

We will often work with induced subgraphs of $\bipartite$, which we will write
as $\bipartitep{I}{J}$, where $I\subset[m]$ and $\overline J\subset[\overline
n]$. These provide the bipartite graph representation of the faces of $\product$
of the form $\productp{I}{J}$, where $\Delta_I:=\{\bfe_i\in\Delta_{m-1}\colon
i\in I\}$. 

\subsection*{Grid representation}

The \defn{$m\times n$ grid}, which we denote by $\grid$, is a rectangular array
of width $m$ and height $n$ composed of $mn$ unit squares. Every unit square in
$\grid$ has a position $(i,\overline j)$ in the grid, where index $i$ increases to the
right and index $\bar j$ increases upwards (i.e., in the usual cartesian way). Thus,
the point $(\bfe_i,\bfe_j)$ in $\product$ is represented by the square at
position $(i,\overline j)$ in $\grid$. The resulting \defn{grid representation} for
subsets of vertices of $\product$ is mainly used in this paper to describe
certain triangulations of $\product$. 

\begin{example}[Staircase triangulation of $\productp{m-1}{n-1}$]
\label{example:staircase}
Consider all monotone paths in $\gr{m}{n}$ from $(1,\overline 1)$ to $(m,\overline n)$. These are
sequences of squares $\{(i_k,\overline j_k)\}_{1\leq k \leq m+n-1}$, such that
${(i_1,\overline j_1)=(1,\overline 1)}$, $(i_{m+n-1},\overline j_{m+n-1})=(m,\overline n)$ and such that
$(i_{k+1},\overline j_{k+1})$ is either $(i_{k}+1,\overline j_{k})$ or $(i_{k},\overline j_{k}+1)$ for $1\leq
k<m+n-1$. Every such monotone path, or staircase, defines a $(m+n-2)$-simplex of
$\product$, whose points are labelled by the squares in the path.

It is easy to see that this collection of simplices forms a triangulation of
$\productp{m-1}{n-1}$, which is called the \defn{staircase triangulation}, and
that it is completely specified by the linear ordering chosen for the
vertices of $\Delta_{m-1}$ and $\Delta_{n-1}$. Even more, one can easily prove
that the staircase triangulation is regular, because it is a pulling
triangulation (cf. \cite[Proposition~6.2.15]{DeLoeraRambauSantos}).
 
\end{example}

\subsection*{Matching ensemble representation}

Every triangulation of $\product$ gives rise to a collection of perfect
matchings on all subgraphs of $\bipartite$ induced by subsets $I\subset [m]$ and
$\overline J\subset [\overline n]$ of the same cardinality. Roughly, it collects
the information of what subset of every circuit of $\product$ appears as a simplex of
the triangulation. 

Recently, Suho Oh and Hwanchul Yoo~\cite{suho_triangulations_2013} have found a
concise characterization of those collections of perfect matchings which
correspond to triangulations of $\product$, hence discovering a novel
\defn{matching ensemble representation} for triangulations of $\product$.

\begin{definition}[{\cite[Definition~4.1]{suho_triangulations_2013}}]\label{def:ensemble}
A family of perfect matchings $\cM$ on all induced subgraphs of $\bipartite$ is
a \defn{matching ensemble} if:
\begin{enumerate}
 \item[\textbf{(SA).}\namedlabel{it:supports}{\textbf{(SA)}}] for each $I\subset
[m]$ and $\overline J\subset [\overline n]$ with $|I|=|\overline J|$, there
exists a unique $\Match \in \cM$ on the subgraph of $\bipartite$ induced by $I$
and $\overline J$ (\defn{supports axiom}),
 \item[\textbf{(CA).}\namedlabel{it:closed}{\textbf{(CA)}}] for each $\Match\in
\cM$ and for each (perfect sub-matching) $\Match'\subseteq \Match$, $\Match'\in
\cM$ (\defn{closure axiom}), and
 \item[\textbf{(LA).}\namedlabel{it:manifold}{\textbf{(LA)}}] for each
$\Match\in \cM$ and for each $v\in[n]\cup[\overline m]\setminus (I\cup \overline
J)$, there are two edges $e'\in \bipartite$ and $e\in \Match$ sharing a common
vertex such that $v \in e'$ and $(\Match\setminus e \cup e')\in \cM$
(\defn{linkage axiom}).
\end{enumerate}
 \end{definition}

\begin{theorem}[{\cite[Theorem~5.4]{suho_triangulations_2013}}]\label{thm:SuhoYoo}
A family of perfect matchings $\cM$ is the family of perfect matchings of a
triangulation $\cT$ of $\product$ if and only if it is a matching ensemble.
\end{theorem}

The lemma below indicates how a triangulation can be recovered from a matching
ensemble:
\begin{lemma}\label{lem:recovertriangulation}
Let $\cM$ be a matching ensemble on $\bipartite$, and let~$\cT$ be its
corresponding triangulation of $\product$. Then a spanning tree $s\in
\bipartite$ represents a (maximal) simplex in $\cT$ if and only if for each
$\Match\in \cM$, there is no cycle in $s\cup \Match$ that alternates between $s$
and $\Match$.
\end{lemma}

\begin{example}
The matching ensemble corresponding to the staircase triangulation of $\product$
from Example~\ref{example:staircase} consists of all non-crossing perfect
matchings on the induced subgraphs of $\bipartite$, where non-crossing means
with non-intersecting edges in the standard embedding of $\bipartite$ in the
plane (see Figure~\ref{fig:staircasematch}).
\begin{figure}[htpb]
\centering
\includegraphics[width=0.2\textwidth]{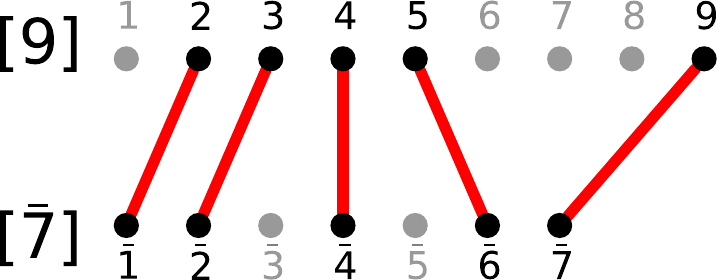}
\caption{A non-crossing perfect matching on $K_{\{2,3,4,5,9\},\{\overline 1,
\overline 2, \overline 4,\overline 6, \overline 7\}}$ that belongs to the
matching ensemble of the staircase triangulation of $\productp{8}{6}$.}
\label{fig:staircasematch}
\end{figure}

\end{example}

Since $\bipartite$ admits a perfect matching if and only if $m=n$, when speaking
of a perfect matching on $\bipartitep{I}{J}$, we will frequently take the
assumption $|I|=|\overline J|$ for granted. This understood, we will also omit
the adjective ``perfect'' whenever there is no risk of confusion.

\subsection*{Mixed subdivisions and tropical arrangements}
 
In order to illustrate some of our constructions, we shall draw triangulations
of $\product$ as \defn{fine mixed subdivisions} of $m\Delta_{n-1}$. Let $\cT$
be a triangulation of $\product$. To each simplex $s\in \cT$ associate the
simplex $s_i=\conv \set{\bfe_j}{(\bfe_i,\bfe_j)\in s}\subset\Delta_{n-1}$. The set of
\emph{Minkowski sums}~$\set{s_1+\dots+s_m}{s\in \cT}$ forms a mixed subdivision of
$m\Delta_{n-1}$. The Cayley trick states that this correspondence is  a
bijection between triangulations of $\product$ and fine mixed subdivisions of
$m\Delta_{n-1}$ (see~\cite{Santos2005} for more details). This provides the link to
an interpretation mentioned in the introduction: the dual of such a
mixed
subdivision can be seen as an arrangement of tropical pseudohyperplanes. In
fact, it turns out that there is a bijection between tropical oriented matroids
and mixed
subdivisions of $m\Delta_{n-1}$ \cite{ArdilaDevelin2009,Horn2012,OhYoo2011}.

\begin{figure}[htpb]
\centering
\includegraphics[width=0.9\textwidth]{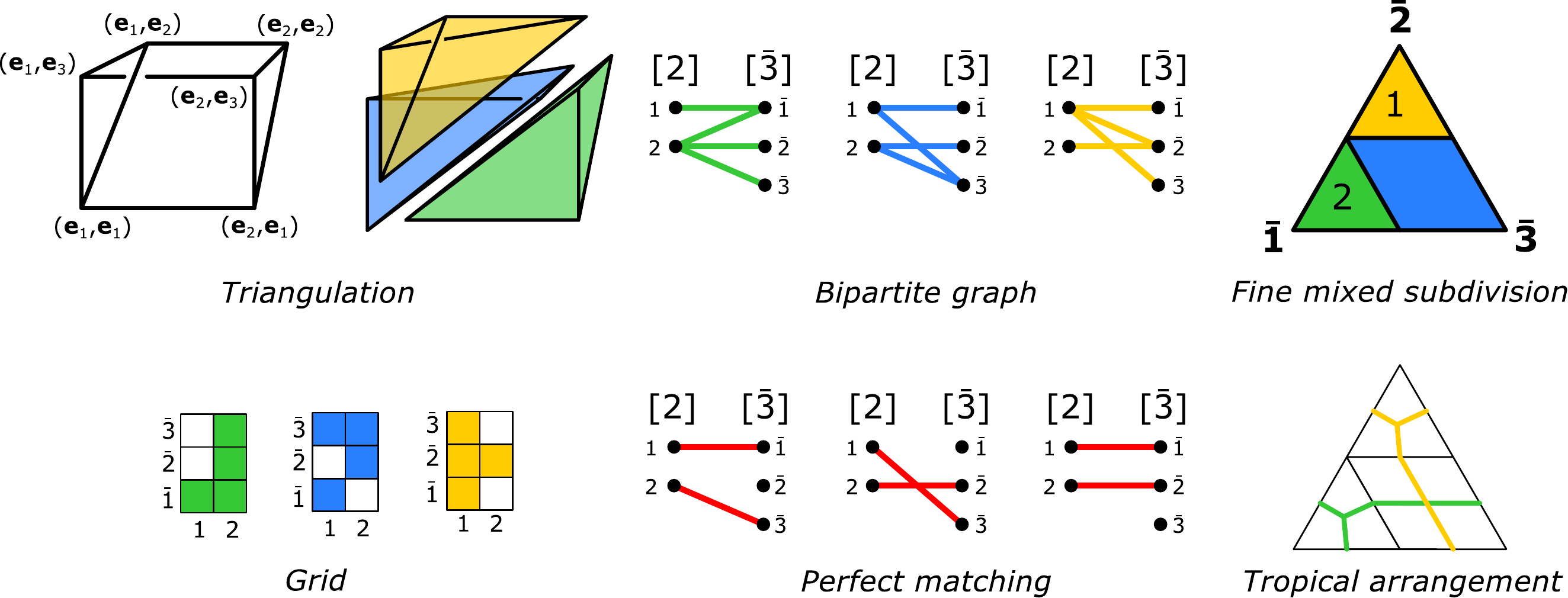} 
\caption{Several representations of a triangulation of $\Delta_1\times\Delta_2$.}
\label{fig:representations}
\end{figure}


\section{Extendable partial triangulations}\label{sec:extend}

In order to present the proof of our extendability Theorem~\ref{theorem:extend}, we first observe that, whenever $m\geq k\geq n$, if a triangulation of $\semiskel$ extends to a triangulation of $\product$, then this extension is unique. We omit the proof, which can be easily deduced from Theorem~\ref{thm:SuhoYoo} (it is also implicit in the proof of a classical result of Dey, cf. \cite[Section 3]{Dey1993} and \cite[Lemma 8.4.1]{DeLoeraRambauSantos}).

\begin{lemma}\label{lem:semiskel-unicity}
Let $m,n,k$ be natural numbers such that $n\geq 2$ and $m\geq k\geq n$. Every triangulation of $\product$ is uniquely determined by its restriction to $\semiskel$.
\end{lemma}

Not every triangulation of~$\semiskel$ can be extended to a triangulation of~$\product$, as is already known from the familiar non-extendable triangulation of ${\Delta_{2}^{(1)}\times\Delta_{1}}$ depicted in Figure~\ref{fig:nonextendable_a}, often called ``the mother of all examples''. However, by strengthening the hypotheses to $k>n$, it becomes possible to certify extendability.

\begin{theorem}\label{theorem:extend}
Let $m,n,k$ be positive integers such that $m\geq k>n$. Every triangulation of $\semiskel$ extends to a unique triangulation of $\product$.
\end{theorem}

\begin{proof}
 Let $\cT'$ be a triangulation of~$\semiskel$ in the conditions of the theorem, and let $\cM$ be the set of all perfect matchings contained in the simplices of $\cT'$, when viewed as subgraphs of $\bipartite$. We prove that $\cM$ is necessarily a matching ensemble (cf. Definition~\ref{def:ensemble})  and hence, by Theorem~\ref{thm:SuhoYoo}, that it is the family of perfect matchings of a triangulation of $\product$. 
  \begin{enumerate}
 \item[\ref{it:supports}.] For each $I\subset [n],\overline J\subset [\overline m]$ with $|I|=|\overline J|\leq n$, the restriction $\cT' |_{\productp{I}{J}}$ is a triangulation of the face $\productp{I}{J}$ of $\product$. Let $\cM |_{\bipartitep{I}{J}}$ be the set of perfect matchings associated to this restricted triangulation.
Since this face is a product of simplices, Theorem~\ref{thm:SuhoYoo} applies, so that there is a unique perfect matching $\Match\in \cM |_{\bipartitep{I}{J}}$ on the induced subgraph $\bipartitep{I}{J}$. 
 \item[\ref{it:closed}.] Again, for each perfect matching $\Match\in\cM$ on $\bipartitep{I}{J}$, $\cT' |_{\productp{I}{J}}$ is a legal triangulation, as are all of its restrictions. In particular, by Theorem~\ref{thm:SuhoYoo}, it follows that $\Match'\in \cM$ whenever $\Match'\subseteq \Match$.
 \item[\ref{it:manifold}.] Fix a perfect matching $\Match\in \cM$ on $\bipartitep{I}{J}$, where $|I|=|\overline J|\leq n$, and let $v\in([m]\cup[\overline n])\setminus([I]\cup[\overline J])$. Observe that $\cT' |_{\productp{I\cup v}{J}}$ (resp. $\cT' |_{\productp{I}{J\cup v}}$) is a legal triangulation, because $k\geq n+1$. In particular, that means that there are two edges $e'\in \bipartitep{I\cup v}{J}$ (resp. $e'\in \bipartitep{I}{J\cup v}$) and $e\in \Match$ sharing a common vertex such that $v \in e'$ and $(\Match\setminus e \cup e')\in \cM$.
 \end{enumerate}

 Finally, the uniqueness of the resulting triangulation was established in Lemma~\ref{lem:semiskel-unicity}.
\end{proof}


\section{The Dyck path triangulation  and some relatives}\label{sec:DyckTriangulation}

There are two main ingredients towards our construction for Theorem~\ref{theorem:non-extend}: the \defn{Dyck path triangulation} and the \defn{extended Dyck path triangulation}. We present them here and explain how they can be used to prove Theorem~\ref{theorem:non-extend}. 

\subsection{The Dyck path triangulation}
The first ingredient is a triangulation of $\productn$ that we dub the \defn{Dyck path triangulation} and denote by $\dyckptri{n}$. 
This triangulation can be described in terms of \defn{Dyck paths} in the grid
representation $\gr{n}{n}$, that is, monotonically increasing paths from the
square $(1,\overline 1)$ to the square $(n,\overline n)$ of $\gr{n}{n}$, in which every square
$(i,\overline j)$ satisfies~$i\leq \overline j$. The maximal simplices of $\dyckptri{n}$ are the
Dyck paths in $\gr{n}{n}$, together with the orbit of simplices they generate
under an action that cyclically shifts the indices in both factors of
$\productn$ simultaneously. Examples for $n=3$ and $n=4$ are depicted in
Figures~\ref{fig:dyck3} and~\ref{fig:dyck4}.
\begin{figure}[htbp]
 \centering
 \begin{subfigure}{0.45\textwidth} \centering
 \includegraphics[width=0.9\textwidth]{dyck3.pdf}  
 \caption{The triangulation $\dyckptri{3}$ of $\productp{2}{2}$.}
 \label{fig:dyck3} 
 \end{subfigure} 
 \begin{subfigure}{0.45\textwidth} \centering
 \includegraphics[width=0.9\textwidth]{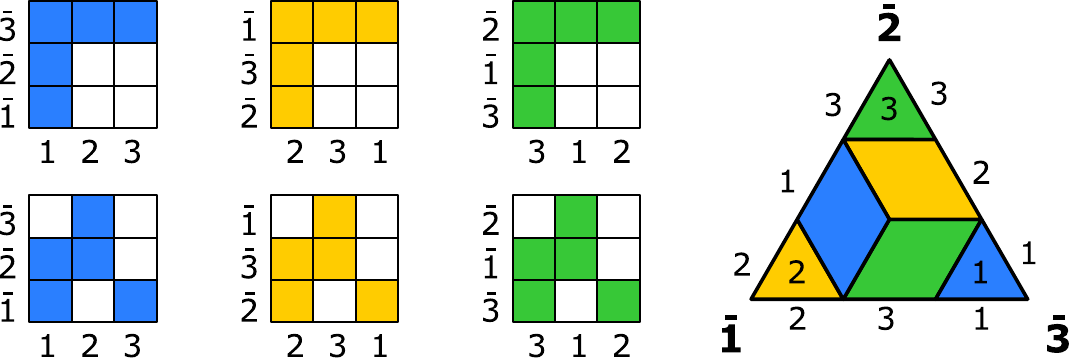}  
 \caption{The triangulation $\dyckptriflip{3}$ of $\productp{2}{2}$.}
 \label{fig:dyck3flip} 
 \end{subfigure} 
 \caption{The Dyck path triangulation of $\productp{2}{2}$ and its flipped
version in the grid and mixed subdivisions representations.}
 \end{figure}
 
 \begin{figure}[htbp]
 \centering
 \includegraphics[width=\textwidth]{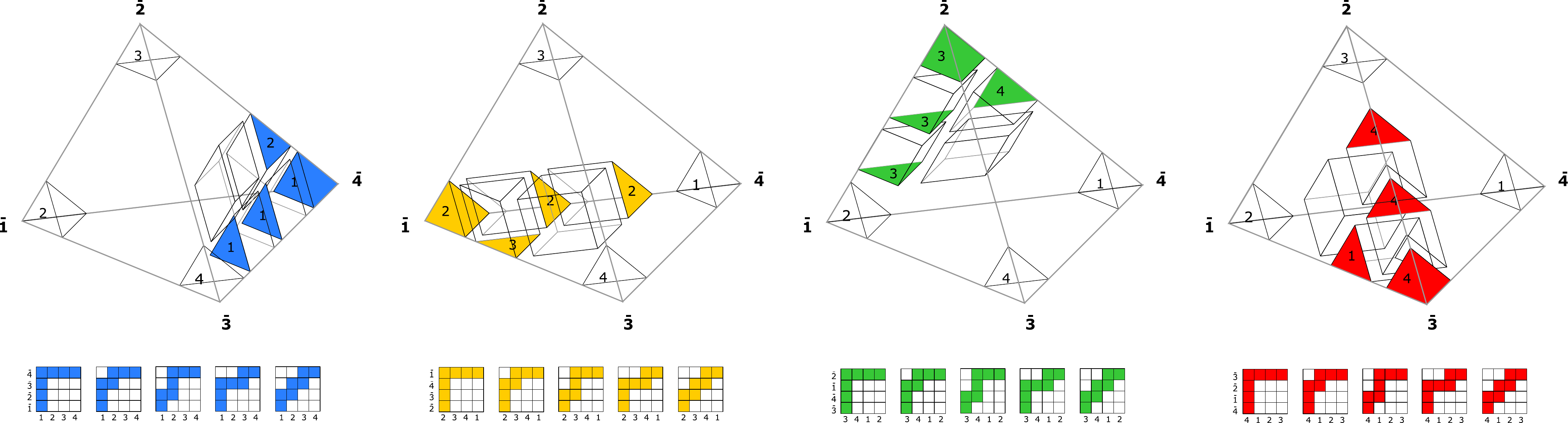}  
 \caption{The triangulation $\dyckptri{4}$ of $\productp{3}{3}$.}
 \label{fig:dyck4} 
 \end{figure}

\begin{theorem}\label{thm:dyck_triangulation}
The Dyck path triangulation is a triangulation of $\productn$. Moreover, it is regular.
\end{theorem}

We defer the proof of Theorem~\ref{thm:dyck_triangulation} to Propositions~\ref{pro:dyckmatch} and~\ref{prop:pushingdyck}.

\begin{remark}
The Dyck path triangulation of $\productn$ is a natural refinement of a coarse regular subdivision introduced by Gelfand, Kapranov and Zelevinsky in~\cite[Example 3.14]{GKZ}. Indeed, the union of all Dyck paths is a cell of this subdivision, and the remaining cells are obtained by applying the cyclic action that shifts the indices.
\end{remark}

For us, the crucial property of $\dyckptri{n}$ that underlies the construction for Theorem~\ref{theorem:non-extend}, is that it admits a \defn{geometric bistellar flip} supported on the circuit $\bfC=(\bfC^+,\bfC^-)$ of maximal dimension given by
\begin{align}\label{eq:circuitflip}
\bfC^+&:=\{(\bfe_1,\bfe_2),(\bfe_2,\bfe_3),\ldots, (\bfe_{n-1},\bfe_n),(\bfe_n,\bfe_1)\} \\ 
\bfC^-&:=\{(\bfe_1,\bfe_1),(\bfe_2,\bfe_2),\ldots, (\bfe_{n-1},\bfe_{n-1}),(\bfe_n,\bfe_n)\}\nonumber.
\end{align}
Therefore, performing this flip consists in replacing only the simplices in $\cT^+:=\{\bfC\setminus\{\bfv\}:\bfv\in\bfC^+\}\subset\dyckptri{n}$ by those in $\cT^-:=\{\bfC\setminus\{\bfv\}:\bfv\in\bfC^-\}$, while leaving the rest of $\dyckptri{n}$ intact; in particular, the flip does not alter the restriction of $\dyckptri{n}$ to the boundary of $\productn$. We refer the reader to~\cite[Section 4.4.1]{DeLoeraRambauSantos} for the precise definition of a geometric bistellar flip. We call the resulting triangulation the \defn{flipped Dyck path triangulation}, and denote it by $\dyckptriflip{n}$; it is illustrated for $n=3$ in Figure~\ref{fig:dyck3flip}. 

\subsection{The extended Dyck path triangulation}

The second ingredient is a natural extension of $\dyckptri{n}$ to a triangulation of $\productp{n}{n-1}$, which we call the \defn{extended Dyck path triangulation} and denote by $\dyckptriext{n}$. In the grid representation, an \defn{extended Dyck path} is formed by several Dyck paths, concatenated one after the other in the grid $\gr{(n+1)}{n}$, and a square in the $(n+1)$-th column and last row of each Dyck path; this is illustrated in Figure~\ref{fig:intersections}. The maximal simplices of $\dyckptriext{n}$ are given by the extended Dyck paths in $\gr{(n+1)}{n}$, together with the orbit of simplices they define under an action that cyclically shifts the indices in both factors of $\productp{[n]}{n-1}\subset\productp{n}{n-1}$ simultaneously (note that here the action ignores the $(n+1)$-th vertex of the first factor).
The simplices of the extended Dyck path triangulation for $n=3$ are shown in Figure~\ref{fig:dyckext3}. Interestingly, the simplices obtained this way constitute a regular triangulation of $\productp{n}{n-1}$.

 \begin{figure}
\centering
\begin{subfigure}{0.3\textwidth}
  \centering  
  \includegraphics[width=0.9\textwidth]{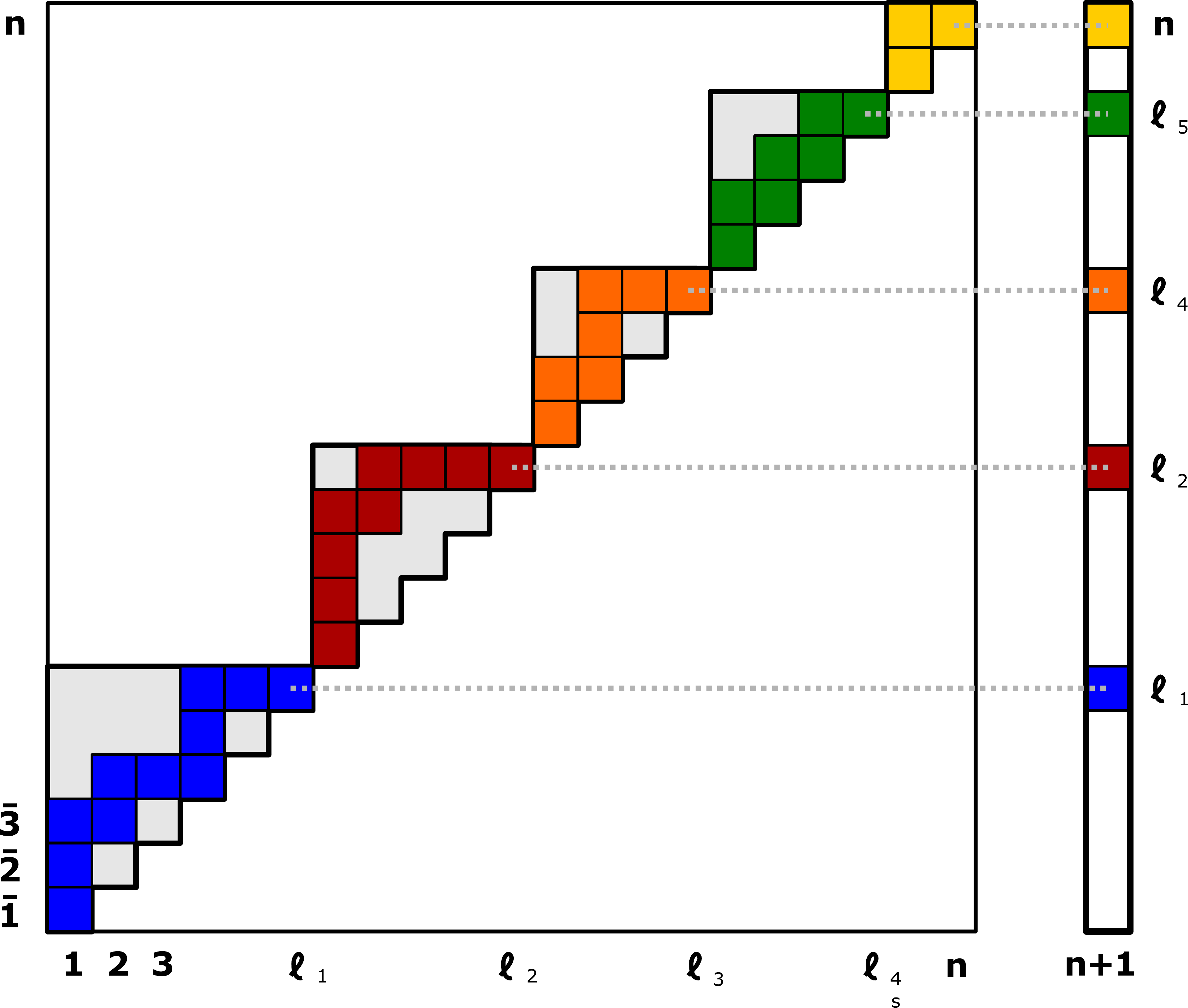}
  \caption{An extended Dyck path, representing a full-dimensional simplex in $\dyckptriext{n}$.}
  \label{fig:intersections}
\end{subfigure}\quad
\begin{subfigure}{0.6\textwidth}
\centering
 \includegraphics[width=0.9\textwidth]{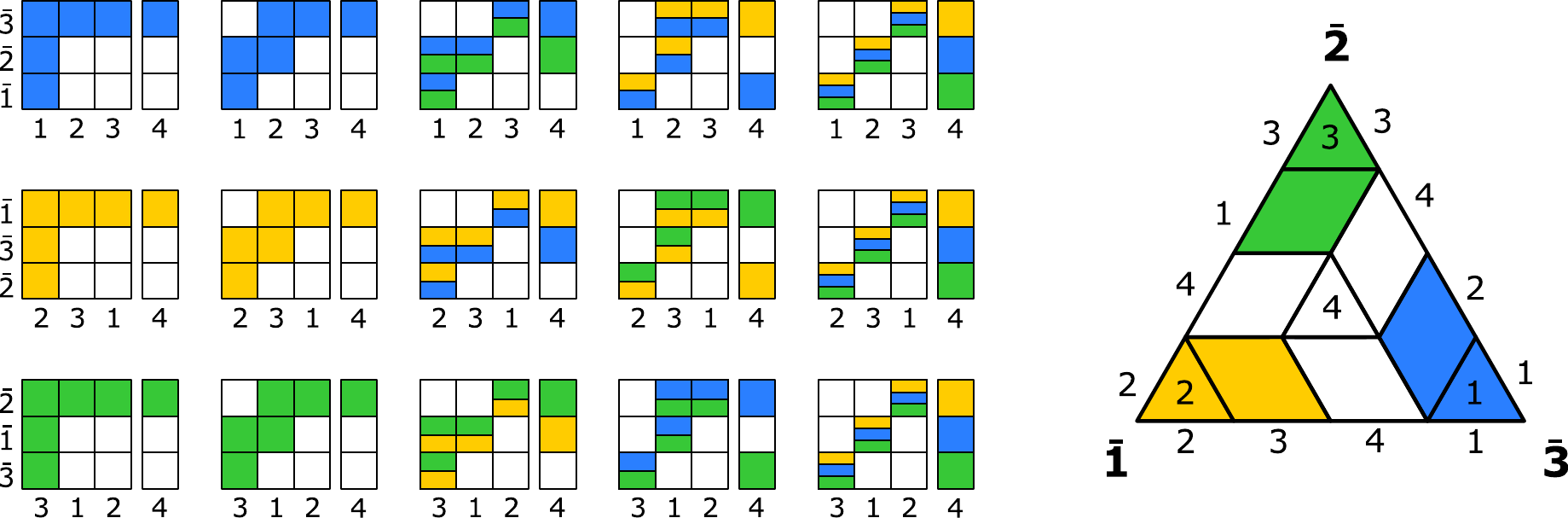}  
\caption{The triangulation $\dyckptriext{3}$ of $\productp{3}{2}$ in the grid and mixed subdivision representations}
\label{fig:dyckext3}
\end{subfigure}
\caption{Illustration of the extended Dyck path triangulation $\dyckptriext{n}$.}
\end{figure}

\begin{theorem}\label{thm:dyck_triangulation_ext}
The extended Dyck path triangulation $\dyckptriext{n}$ is a triangulation of
$\productp{n}{n-1}$. Moreover, it is regular.
\end{theorem}

The proof of Theorem~\ref{thm:dyck_triangulation_ext} is addressed in
Propositions~\ref{pro:extdyckmatch} and~\ref{prop:extendeddyckisregular}.

\begin{remark}
Substituting the word ``square'' by ``edge'' in the definitions of
$\dyckptri{n}$ and $\dyckptriext{n}$, we automatically get their descriptions in
the bipartite graph representation. Thus, the maximal simplices of
$\dyckptri{n}$ are represented in $\bipartitep{n}{n}$ by non-crossing and weakly
increasing spanning trees, plus the orbit of spanning trees they generate under
the action that cyclically shifts the indices of $[n]$ and $[\overline n]$ simultaneously, as shown in Figure~\ref{fig:ncwi_trees1}.

Likewise, the maximal simplices of $\dyckptriext{n}$ consist of concatenated non-crossing and weakly increasing spanning trees on subgraphs of $\bipartitep{[n]}{n}$ with an edge between $n+1$ and the last vertex of $[\overline n]$ in every subgraph, along with their images under the cyclic
shift of the indices $[n]\subset[n+1]$ and $[\overline n]$. This is best understood looking at Figure~\ref{fig:ncwi_trees2}.
\begin{figure}[htbp]
 \centering
 \begin{subfigure}{0.4\textwidth} \centering
 \includegraphics[width=0.8\textwidth]{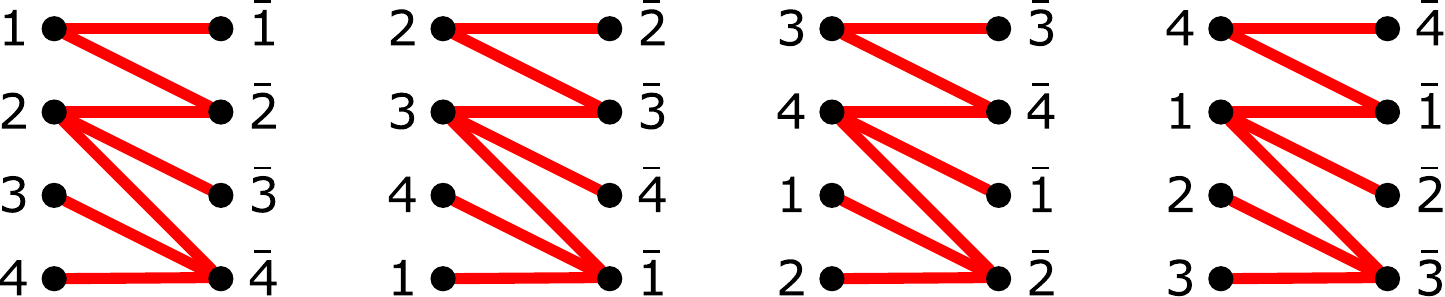}  
 \caption{A non-crossing and weakly increasing spanning tree in
$\bipartitep{4}{4}$, and its orbit under cyclic shifting.}
\label{fig:ncwi_trees1}
 \end{subfigure} \quad
 \begin{subfigure}{0.4\textwidth} \centering
 \includegraphics[width=\textwidth]{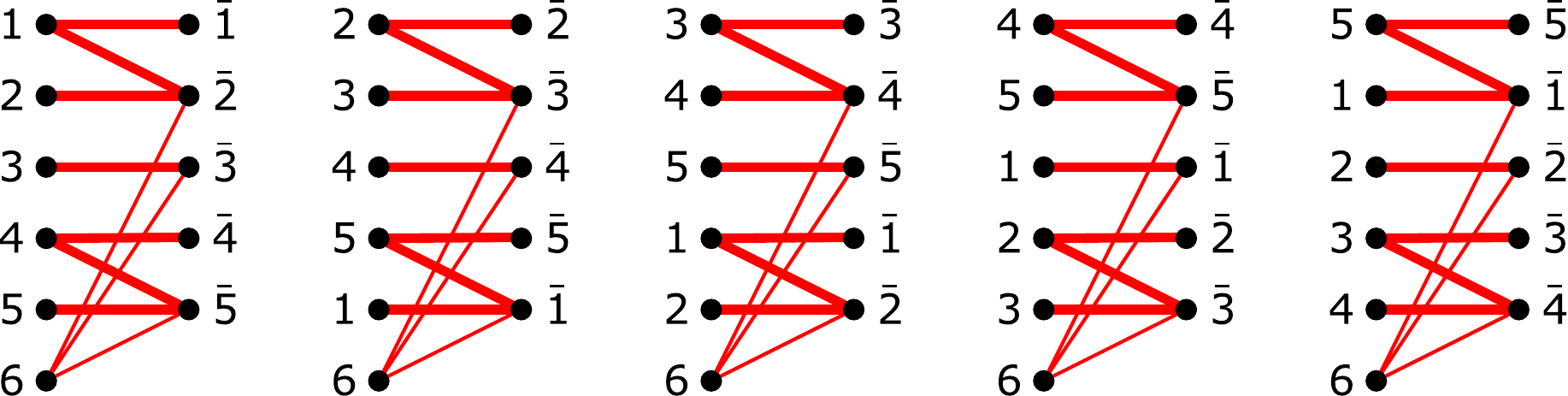}  
 \caption{Bipartite graph representation of an orbit of maximal simplices in
$\dyckptriext{4}$.}
\label{fig:ncwi_trees2}
 \end{subfigure} 
 \caption{Illustration of the bipartite graph representation of $\dyckptri{n}$
and $\dyckptriext{n}$.}
 \label{fig:ncwi_trees}
 \end{figure}
\end{remark}

The restriction of $\dyckptriext{n}$ to $\semibdy{n}{n-1}$ gives a partial
triangulation of $\productp{n}{n-1}$ whose restriction to the facet
$\productp{[n]}{n-1}$ coincides with $\dyckptri{n}$. Denote by
$\bdydyckptriext{n}$ this restricted triangulation, whose facet
$\productp{[n]}{n-1}$ admits a bistellar flip supported on the
circuit~\eqref{eq:circuitflip}. By this, we mean that the triangulation of the
facet $\productp{[n]}{n-1}$ can be changed from $\dyckptri{n}$ to
$\dyckptriflip{n}$, without affecting the triangulation of the remaining facets
of $\semibdy{n}{n-1}$. Hence, the result of the flip is still a triangulation
$\semibdy{n}{n-1}$, which we call \defn{flipped extended Dyck path
triangulation} and denote $\bdydyckptriflip{n}$. An example is depicted in
Figure~\ref{fig:non-extendable}. 

 \begin{figure}[htbp]
 \centering
 \includegraphics[width=0.8\textwidth]{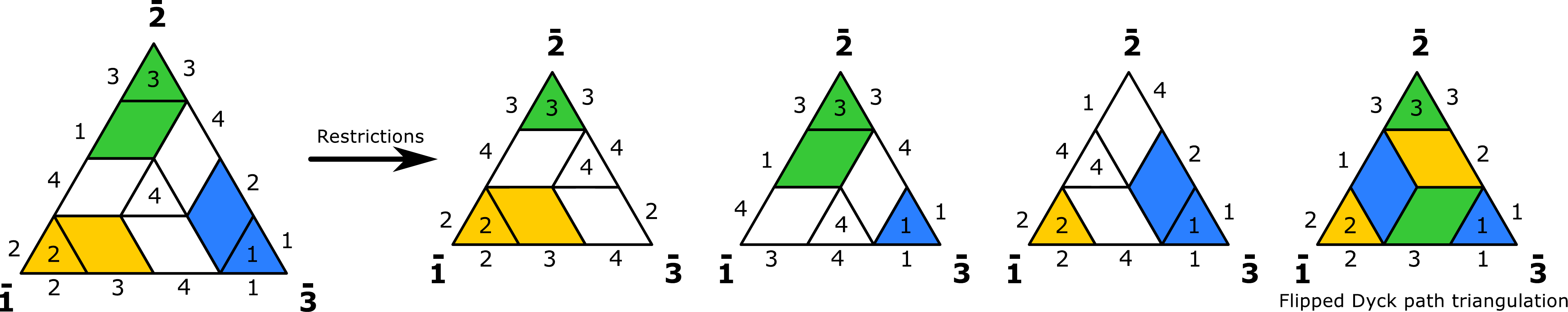}  
 \caption{A non-extendable $\protect \bdydyckptriflip{n}$ triangulation of $\semibdy{3}{3-1}$.}
 \label{fig:non-extendable} 
 \end{figure} 

\begin{theorem}\label{theorem:non-extend}
The flipped extended Dyck path triangulation $\bdydyckptriflip{n}$ of $\semibdy{n}{n-1}$ is non-extendable.
\end{theorem}

Our proof of Theorem~\ref{theorem:non-extend} uses also the language of matching
ensembles, and is delayed until Section~\ref{sec:matchings}. 

\begin{corollary}
For every $m\geq n$ there exist non-extendable triangulations of
$\semiskelp{m}{n-1}{n-1}$.
\end{corollary}
\begin{proof}
The same construction that produces $\bdydyckptriflip{n}$ from
$\dyckptriext{n}$ can be applied on any triangulation of $\productp{m}{n-1}$
that restricts to $\dyckptriext{n}$ on a face. Such a triangulation is
easy to construct from $\dyckptriext{n}$, for example, by doing a pushing
refinement (cf. \cite[Lemma 4.3.2]{DeLoeraRambauSantos}).
\end{proof}


\section{Perfect matching representations}\label{sec:matchings}

In this section, we present the proofs of Theorem~\ref{thm:dyck_triangulation} and Theorem~\ref{thm:dyck_triangulation_ext}, which assert that the Dyck path triangulation $\dyckptri{n}$ and its extension $\dyckptriext{n}$ are indeed triangulations, along with the proof of Theorem~\ref{theorem:non-extend}. All are phrased in terms of the matching ensemble representation for triangulations of $\product$, characterized in Theorem~\ref{thm:SuhoYoo}. Therefore, our first task is to describe the matching ensembles arising from $\dyckptri{n}$ and $\dyckptriext{n}$, which closely resembles the construction of $\dyckptri{n}$ and of $\dyckptriext{n}$.

\subsection{Matching ensemble of the Dyck path triangulation $\dyckptri{n}$}
We begin with the set of all perfect matchings on the subgraphs of $\bipartitep{n}{ n}$ induced by $I\subset [n]$, $\overline J\subset[\overline n]$ (with $|I|=|\overline J|$) which are non-crossing (nc) and weakly increasing (wi), that is, those matchings $\Match$ that satisfy
\begin{equation}\tag{\footnotesize \textbf{nc+wi}} \label{eq:incmatching}
\begin{cases}
i< i' \Rightarrow \overline j< \overline j ' \text{ for every } (i,\overline j),(i',\overline j')\in\Match, \\
i\leq \overline j \text{ for every } (i,\overline j)\in \Match
\end{cases}.
\end{equation}
Next, for $\ell\in[n]$, we introduce the collection of matchings of the form
\begin{equation} \tag{\footnotesize \textbf{cyc}} \label{eq:shiftmatch}
\left\{ \Big(i+\ell \pmod n,\ \overline j+\ell\pmod n\Big)\colon (i,\overline j)\in\Match,\ \Match\text{  fulfills~\eqref{eq:incmatching}} \right\},
\end{equation}
obtained by ``cyclically shifting'' the indices of the perfect matchings that satisfy~\eqref{eq:incmatching}, and call $\cM_n$ the set of all matchings obtained after ranging over all $\ell\in[n]$ (see Figure~\ref{fig:matchandorbit}).

\begin{figure}[htbp]
\begin{center}
\includegraphics[width=0.6\textwidth]{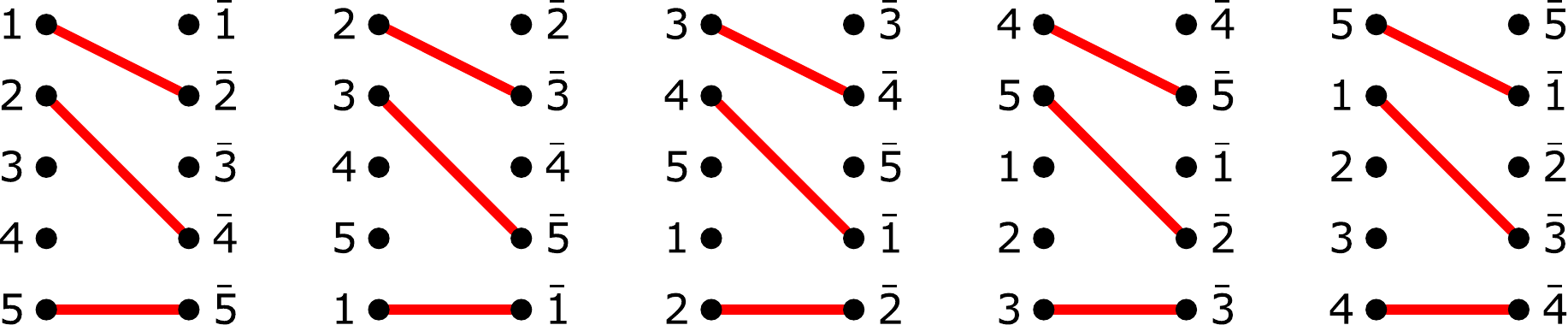}
\caption{A perfect matching on $K_{\{1,2,5\},\{\overline 2,\overline 4, \overline 5\}}$ satisfying~\eqref{eq:incmatching}, together with its orbit of perfect matchings gotten as~\eqref{eq:shiftmatch}.}
\label{fig:matchandorbit}
\end{center}
\end{figure}
\begin{proposition} \label{pro:dyckmatchensemble}
The collection of matchings $\cM_{n}$ constitutes a matching ensemble.
\end{proposition}

\begin{proof}
Proof of~\ref{it:supports}. View the elements of $[n]\cup[\overline n]$ as a totally ordered string that extends the order of $[n]$ and $[\overline n]$ with $i\prec \overline j$ whenever $i\leq\overline j$ for $i\in [n]$, $\overline j\in [\overline n]$. For every $I\subset[n]$ and $\overline J\subset[\overline n]$ with $|I|=|\overline J|$, accordingly view the subset $I\cup \overline J$ as a totally ordered substring of $[n]\cup[\overline n]$.

To define a matching $\Match'$ on $K_{I,\overline J}$, ``cyclically rotate'' the ordering $\prec$ (by putting the first elements last) so that all the final substrings of $I\cup \overline J$, consisting of the last elements in the string $I\cup \overline J$, have at least as many elements from $\overline J$ as from $I$. Then, $\Match'$ is gotten by matching the $k$-th element of $I$ with the $k$-th element of $\overline J$ in the rotated string $I\cup \overline J$. It is easy to check that, given $I$ and $\overline J$, this rule uniquely determines $\Match'$; it is illustrated in Figure~\ref{fig:matchingproof1}.

Denoting the first element of the rotated string $I\cup \overline J$ by $\ell+1$ (which belongs to $I$), we see that the resulting $\Match'$ has the form~\eqref{eq:shiftmatch}. Conversely, all matchings of the form~\eqref{eq:shiftmatch} can be obtained as explained above.

\begin{figure}[htbp]
 \centering
 \begin{subfigure}{0.3\textwidth} \centering
 \includegraphics[width=\textwidth]{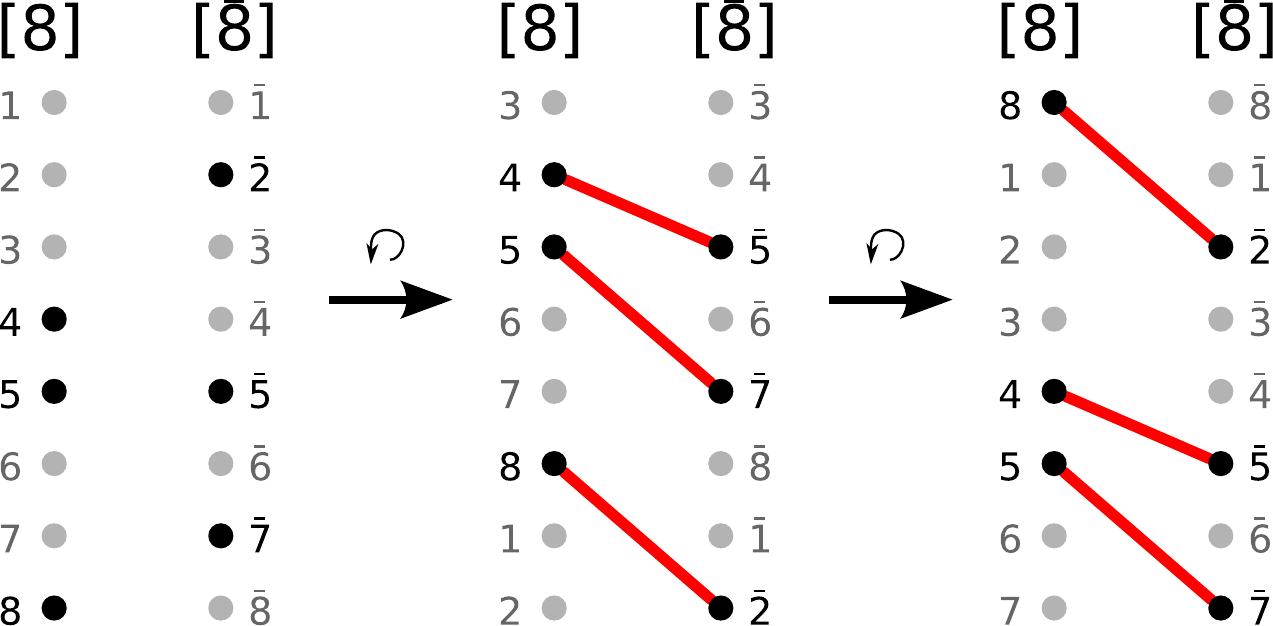}  
 \caption{Retrieving the unique perfect matching on $K_{\{4,5,8\},\{\overline 2, \overline 5, \overline 7\}}$.}
 \label{fig:matchingproof1}
 \end{subfigure} \quad
 \begin{subfigure}{0.6\textwidth} \centering
 \includegraphics[width=0.85\textwidth]{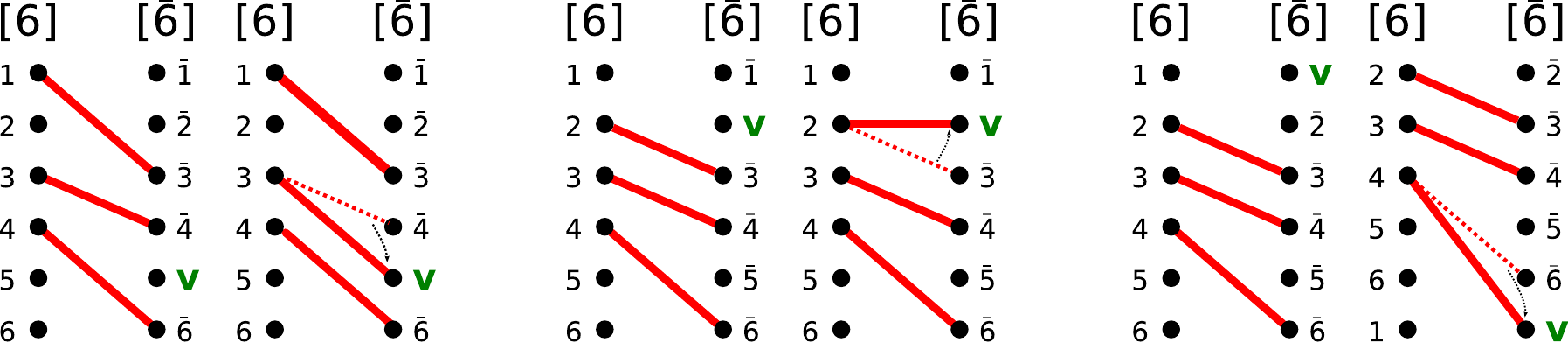}  
 \caption{Three possible situations when checking~\ref{it:manifold} for the collection of matchings $\cM_n$.}
 \label{fig:matchingproof2}
 \end{subfigure} 
 \caption{Proof of Proposition~\ref{pro:dyckmatch}}
 \label{fig:matchingproof}
 \end{figure}

Proof of~\ref{it:closed}. Clearly, if $\Match'$ is a perfect sub-matching of the perfect matching $\Match\in \cM_n$ on $\bipartitep{I}{ J}$, it is still of the form~\eqref{eq:shiftmatch}, so $\Match'\in\cM_n$.

Proof of~\ref{it:manifold}. Assume $\Match$ is a perfect matching on $\bipartitep{I}{ J}$ satisfying~\eqref{eq:incmatching}, and let $v\in [\overline n]\setminus \overline J$ (the general case follows by rotation and symmetry). If there is some $(i,\overline j)\in\Match$ with $v > \overline j$, set \linebreak $\overline j_0:=\max\{ \overline j<v \colon (i,\overline j )\in\Match\}$ and define
\[
\Match'=\Match\setminus (i_0,\overline j_0)\cup(i_0,v).
\]
On the other hand, if $v< \overline j$ for every $(i,\overline j)\in \Match$, set $i_1:=\min\{i \colon (i,\overline j)\in \Match\}$, ${i_2:=\max\{i\colon (i,\overline j)\in\Match\}}$, and define:
\[
\Match':=\begin{cases}
\Match\setminus (i_1,\overline j_1)\cup (i_1,v) & \text{if }v\geq i_1\\
\Match\setminus (i_2,\overline j_2)\cup (i_2,v) & \text{otherwise}
\end{cases}
\]

Either way, $\Match'$ is a perfect matching on $\bipartitep{I}{ J\setminus \overline j \cup v}$ obtained as~\eqref{eq:shiftmatch}; hence $\Match'\in\cM_n$. The three cases are drawn in Figure~\ref{fig:matchingproof2}.
\end{proof}

This settles the proof that the Dyck path triangulation is indeed a triangulation.

\begin{proposition} \label{pro:dyckmatch}
The Dyck path triangulation $\dyckptri{n}$ is the triangulation associated to the matching ensemble $\cM_{n}$.
\end{proposition}
\begin{proof}
 By Lemma~\ref{lem:recovertriangulation}, we need to check that there is no circuit $\bfC$ in $s\cup \Match$ alternating between $s$ and~$\Match$, for a simplex $s\in \dyckptri{n}$ and a matching $\Match\in\cM_n$. If $\bfC$ existed,  then there would be matchings $\Match_1\subset s$ and $\Match_2\subset \Match$ that have the same support. However, it is straightforward to check that every matching $\Match_1\subset s$ belongs to $\cM_n$ by construction, and hence fulfills Axiom~\ref{it:supports}.
This shows that every simplex $s \in \dyckptri{n}$ is a simplex in the triangulation associated to the matching ensemble~$\cM_n$. On the other hand, no further simplices belong to the triangulation associated to the matching ensemble~$\cM_n$, for $\dyckptri{n}$ already exhausts the $n\catnum{n-1}=\binom{2n-2}{n-1}$ full-dimensional simplices every triangulation of~$\Delta_{n-1}\times\Delta_{n-1}$ has. 
\end{proof}

\subsection{Matching ensemble of the extended Dyck path triangulation $\dyckptriext{n}$}

Now we start with the set of matchings $\Match$ between $I\subset[n+1]$ and $\overline J\subset[\overline n]$ with the property
\begin{equation}\label{eq:incextmatching}
\begin{cases}
i< i' \Rightarrow \overline j< \overline j' \text{ for every } (i,\overline j),(i',\overline j')\in\Match 
, \\
i\leq \overline j \text{ for every } (i,\overline j)\in \Match \text{ with } i\neq n+1.
\end{cases} \tag{\ref*{eq:incmatching}$^\text{ext}$}
\end{equation}
As before, we consider the set of perfect matchings on induced subgraphs of $\bipartitep{n+1}{ n}$ of the form

\begin{align} \label{eq:shiftextmatch}
&\left\{ \Big(\rho_\ell(i),\ \overline j+\ell\pmod n\Big)\colon (i,\overline j)\in\Match,\ \Match\text{ fulfills~\eqref{eq:incextmatching}} \right\}, \tag{\ref*{eq:shiftmatch}$^\text{ext}$} \\ \text{where }\rho_\ell(i):=&\begin{cases} i+\ell \pmod n &\text{if }i\neq n+1 \\ n+1 & \text{otherwise} \end{cases} \nonumber,
\end{align}
with $\ell$ ranging over $[n]$, and denote it by $\Mext_n$.

\begin{proposition}\label{pro:extdyckmatchensemble}
The collection $\Mext_n$ of matchings constitutes a matching ensemble.
 \end{proposition}

\begin{proof}
We only have to verify the conditions in Definition~\ref{def:ensemble} when $n+1\in[n+1]$ gets involved; the remaining cases have already been dealt with in Proposition~\ref{pro:dyckmatchensemble}. 

Proof of~\ref{it:supports}. Let $I\subset[n+1]$ and $\overline J\subset[\overline n]$ with $n+1\in I$ and $|I|=|\overline J|$. We order $[n]\cup [\overline n]$ again as in the proof of Proposition~\ref{pro:dyckmatchensemble}, and consider the substring $I'\cup \overline J$, where $I'=I\setminus n+1$. This time, we cyclically rotate the order $\prec$ so that all final substrings of $I'\cup \overline J$ have \emph{strictly more} elements from $\overline J$ than from $I$ (thereby, in particular, the ordering of the substring $I'\cup \overline J$ becomes fixed). 

Let $\Match$ be the perfect matching on $\bipartitep{I}{ J}$ that pairs the $k$-th element of $I'$ with the $k$-th element of~$\overline J$ in the rotated string $I'\cup\overline J$, and $n+1$ with the \emph{unpaired} last element from $\overline J$ (cf. Figure~\ref{fig:extmatchproof1}). This yields a unique matching on $\bipartitep{I}{ J}$ of the form~\eqref{eq:shiftextmatch}. Conversely, all perfect matchings on induced subgraphs of $\bipartitep{n+1}{ n}$ of the form~\eqref{eq:shiftextmatch} can be obtained with this rule.

\begin{figure}[htbp]
 \centering
 \begin{subfigure}{0.2\textwidth} \centering
 \includegraphics[width=0.8\textwidth]{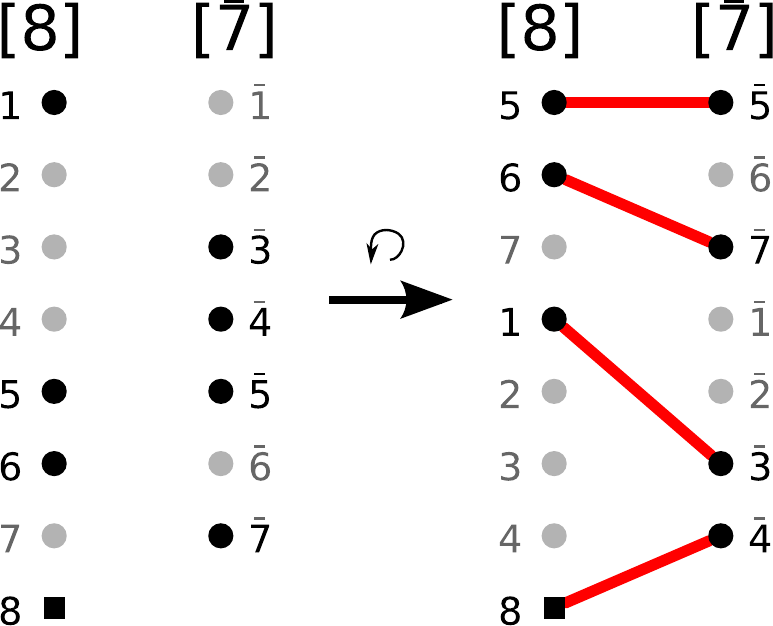}  
 \caption{\footnotesize Retrieving the unique perfect matching on $K_{\{1,5,6,8\},\{\overline 3, \overline 4, \overline 5, \overline 7\}}$.}
 \label{fig:extmatchproof1}
 \end{subfigure} \quad
 \begin{subfigure}{0.75\textwidth} \centering
 \includegraphics[width=0.85\textwidth]{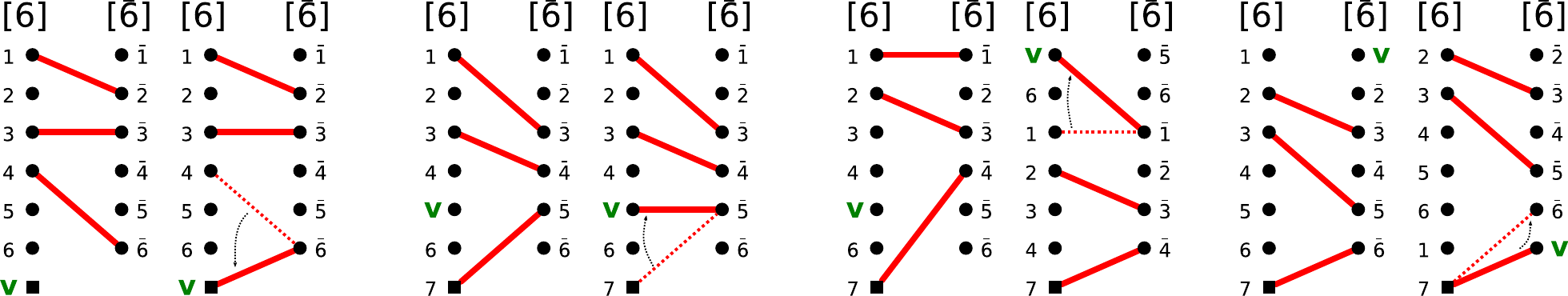}  
 \caption{Some of the possible situations when checking~\ref{it:manifold} for the collection of matchings $\Mext_n$.}
 \label{fig:extmatchproof2}
 \end{subfigure} 
 \caption{Proof of Proposition~\ref{pro:extdyckmatch}}
 \label{fig:extmatchproof}
 \end{figure}

Proof of~\ref{it:closed}. If $\Match$ fulfills \eqref{eq:shiftextmatch} then, trivially, so do all its perfect sub-matchings.

Proof of~\ref{it:manifold}. Let $\Match$ be a perfect matching on $\bipartitep{I}{ J}$ for which~\eqref{eq:incextmatching} holds, and let $v\in([n+1]\cup[\overline n])\setminus( I\cup\overline J)$. We distinguish several cases, that we have depicted in Figure~\ref{fig:extmatchproof2}:
\begin{enumerate}
\item $v=n+1$: set $i'=\max\{i\colon (i,\overline j)\in \Match\}$, then $\Match\setminus (\overline i',\overline j')\cup (v,\overline j')$ also satisfies~\eqref{eq:incextmatching}.
\item $v\in [n+1]$, $v\neq n+1$ and $v> i$ for all $(i,\overline j)\in\Match$: write $(n+1,\overline j^*)\in\Match$, then either
\begin{itemize}
\item $v\leq \overline j^*$ and $\Match\setminus (n+1,\overline j^*)\cup (v,\overline j^*)$ is of the form~\eqref{eq:incextmatching}, or
\item $v>\overline j^*$, in which case $\Match\setminus (i_1,\overline j_1)\cup(v,\overline j_1)$ is gotten as~\eqref{eq:shiftextmatch}, where $i_1:=\min\{i\colon (i,\overline j)\in\Match\}$.
\end{itemize} 
\item \label{it:checkextmanifold} $v \in [\overline n]$ and either $v< i_1$ or $v>\overline j_2$, where $i_1 :=\min\{i\colon (i,\overline j)\in\Match\}$ and $\overline j_2:=\max\{j\colon (i,\overline j)\}$: here $\Match\setminus (n+1,\overline j^*)\cup (n+1,v)\in\Mext_n$, with $\overline j^*$ as previously defined.
\end{enumerate}
The verification of the remaining cases
\begin{itemize}
\item[2'.] $v\in [n+1]$, $v\neq n+1$ and $v< i$ for some $(i,\overline j)\in\Match$,
\item[3'.] $v \in [\overline n]$ and $i_1 < v < \overline j_2$, where $i_1,\overline j_2$ are as in~\ref{it:checkextmanifold}. above,
\end{itemize}
 is skipped, for these do not involve $n+1$ and thus reduce to the situation of Proposition~\ref{pro:dyckmatchensemble}.
\end{proof}

We omit the proof of the following proposition, that shows that $\dyckptriext{n}$ is a triangulation, because it is analogous to that of Proposition~\ref{pro:dyckmatch}. Indeed, one can see that the perfect matchings contained in every simplex of $\dyckptriext{n}$ 
can be rotated to satisfy conditions~\eqref{eq:incextmatching}, and that there are precisely~$\binom{2n-1}{n-1}$ full-dimensional simplices in $\dyckptriext{n}$, as in every triangulation of $\productp{n}{n-1}$.

\begin{proposition}\label{pro:extdyckmatch}
The extended Dyck path triangulation $\dyckptriext{n}$ is the triangulation corresponding to the matching ensemble $\Mext_n$.
 \end{proposition}
 
 We conclude this section with the promised proof of Theorem~\ref{theorem:non-extend}.
 
 \begin{proof}[Proof of Theorem~\ref{theorem:non-extend}]
The triangulation $\bdydyckptriflip{n}$ of $\semibdy{n}{n-1}$ produces a collection of perfect matchings on all induced subgraphs of $\bipartitep{n+1}{n}$, that we refer to as $\cM'$ (for which the reader can check that the axioms~\ref{it:supports} and~\ref{it:closed} hold). Observe that, by construction, $\Mext_n$ and $\cM'$ agree on all the induced subgraphs $K_{I,[\overline n]}$, where $n\notin I\subset[n+1]$. In contrast, the triangulation $\dyckptriflip{n}$ of $\productn$ contributes the following matching on the induced subgraph $K_{[n],[\overline n]}$:
\[
\Match:=\{(1,\overline 2),(2,\overline 3),(3,\overline 4),\ldots, (n-1,\overline n),(n,\overline 1)\}.
\]

Suppose, for the sake of absurdity, that axiom~\ref{it:manifold} holds for $\Match \in \cM'$. Then, there is a unique perfect matching $\Match'\in\cM'$ on $K_{[n]\setminus w \cup n+1, [\overline n]}$ that differs from $\Match$ in a single edge. However, letting $w=n\in[n]$ (which we may by symmetry), we see that the unique perfect matching on $K_{[n]\setminus w\cup n+1,[\overline n]}$ in the matching ensemble $\Mext_n$ is
\[
\{(1,\overline 1),(2,\overline 2),(3,\overline 3),\ldots, (n-1,\overline{n-1}), (n+1,\overline n)\},
\]
so $\Match$ cannot satisfy axiom~\ref{it:manifold}, $\cM'$ is not a matching ensemble and $\bdydyckptriflip{n}$ cannot be extended to a triangulation of $\productp{n}{n-1}$.
\end{proof}


\section{Proof of regularity}\label{sec:regularity}

We have already seen that $\dyckptri{n}$ and $\dyckptriext{n}$ are triangulations of $\productn$ and $\productp{n}{n-1}$, respectively. In this section we prove that they are also regular. We refer to \cite{DeLoeraRambauSantos} for the definitions of regular triangulation, height function, pushing triangulation, etc.

\begin{proposition}\label{prop:pushingdyck}
 The Dyck path triangulation $\dyckptri{n}$ of $\productp{n-1}{n-1}$ coincides with its pushing triangulation 
with respect to any order of the boxes in the grid that extends the partial order:
\[
 (i,\overline j)<(i',\overline j') \quad  \Leftrightarrow \quad  j-i\pmod n <  j'-i' \pmod n, 
\] 
where $ j-i \pmod n$ and $ j'-i' \pmod n$ are taken in $[n]$. 

Hence, the triangulation $\dyckptri{n}$ is regular, and can be obtained by the height function $h\colon \productn\to \RR$ that assigns to the point $(\bfe_i,\bfe_{j})$ the height $h_{ij}=c^{j-i \pmod n}$, for some real number $c>1$ sufficiently large.
\end{proposition}

We omit the proof of Proposition~\ref{prop:pushingdyck}, because it is a direct consequence of Proposition~\ref{prop:extendeddyckisregular} below.

\begin{proposition}\label{prop:extendeddyckisregular}
 The extended Dyck path triangulation $\dyckptriext{n}$ of $\productp{n}{n-1}$ is regular, obtained by assigning the height $h_{ij}$ to the point $(\bfe_i,\bfe_{j})$, defined by:
 \begin{equation}\label{eq:heights}
 h_{ij}=
  \begin{cases}
	  c^{j-i}& \text{ if } j\geq i \\
	  c^{n+j-i}&\text{ if } j <i <n+1\\
	  1&\text{ if } i=n+1;
  \end{cases}
 \end{equation}
where $c>1$ is a large enough real number.
\end{proposition}

To prove this, we use the following result. It is a direct consequence of \cite[Theorem 2.3.20 and Lemma 2.4.2]{DeLoeraRambauSantos}, restricted to the special case of the product of two simplices and expressed in terms of perfect matchings.
\begin{lemma}\label{lem:regularmatching}
Let $\cT$ be the regular subdivision of $\product$ induced by the height function that maps $(\bfe_i,\bfe_{j})$ onto $h_{ij}$.
Then $\cT$ is a triangulation with matching ensemble $\cM$ if and only if for any perfect matching $\Match\in\cM$ on $\bipartitep{I}{J}$ it holds
\begin{equation}
\sum_{(i,\ol j)\in\Match} h_{ij}<\sum_{(i,\ol j)\in\Match'} h_{ij},
 \label{equation:regularmatching}
\end{equation} 
whenever $\Match'\neq\Match$ is a perfect matching on $\bipartitep{I}{J}$.
\end{lemma}

\begin{proof}[Proof of Proposition~\ref{prop:extendeddyckisregular}]
 Fix $I=\{i_1<\cdots<i_s\}\subseteq [n+1]$ and $\ol J=\{\ol j_1<\cdots<\ol j_s\}\subseteq [\ol n]$, and let $\Match$ be the perfect matching in $\bipartitep{I}{J}$ that minimizes $\omega(\Match)$, where we abbreviate $\omega(\Match):=\sum_{(i,\ol j)\in\Match} h_{ij}$. We claim that $\Match\in\Mext_n$.
 
Using the symmetry in the definition of $h_{ij}$ we observe that $\omega(\Match)=\omega(\Match')$ whenever $\Match'$ is obtained from $\Match$ by changing every $(i,\ol j)\in\Match$ by $\big(\rho_\ell(i),\ \overline j+\ell\pmod n\big)$. Therefore, without loss of generality we can shift $I$ and $\ol J$ and always assume that $i_k\leq \ol j_k$ for all $k$ with $i_k\neq n+1$ (compare the proof of Proposition~\ref{pro:extdyckmatchensemble}).
 
Therefore we only need to show that $\Match$ is non-crossing. In our setting, $\Match$ has a crossing if and only if it contains an edge $(i_k,\ol j_\ell)\in \Match$ with $k>\ell$. The proof is by induction on $s=|I|=|\ol J|$ and if $s=1$ then it is trivially true. For $s>1$, let $\ell$ be such that $(i_s,\ol j_\ell)\in \Match$. Then $\Match$ induces a submatching in $I\setminus i_s, \ol J\setminus \ol j_\ell$ that still fulfills $i_k\leq \ol j_k$. By induction hypothesis this submatching must be non-crossing. Hence $\Match$ must be of the form $\Match=\bigcup_{1\leq k<\ell} (i_{k},\ol j_{k})\cup \bigcup_{\ell\leq k<s} (i_{k},\ol j_{k+1})\cup (i_s, \ol j_\ell)$. If $\ell=s$ then $\Match$ is non-crossing as desired.

On the contrary, if $\ell\neq s$, define $\Match':=\bigcup_{1\leq k\leq s} (i_{k},\ol j_{k})$. We claim that for every $k$ there is an edge $(i,\overline j)\in \Match$ such that $h_{i_{k}j_{k}}\leq h_{ij}$ (and strict inequality for at least one $k$). Indeed, 
\begin{itemize}
 \item for $1\leq k<\ell$, there is nothing to prove because $\Match$ and $\Match'$ coincide;
 \item if $\ell\leq k<s$ then $h_{i_{k} j_{k}}<h_{i_{k} j_{k+1}}$ because $i_{k}\leq j_{k}< j_{k+1}$;
 \item for $k=s$, if $i_s\neq n+1$ then $h_{i_{s} j_{s}}< h_{i_{s-1} j_{s}}$ because $i_{s-1}<i_{s}\leq \ol j_{s}$;
 \item finally, if $i_s= n+1$, then $h_{i_s j_{s}}= h_{i_{s} j_{\ell}}$ by definition. 
\end{itemize}
To conclude the proof we just need to observe that when $c$ is large enough then $\max_{(i,\ol j)\in\Match'} h_{i j}<\max_{(i,\ol j)\in\Match} h_{i j}$ implies that $\omega(\Match')<\omega(\Match)$, which contradicts the assumption of $\Match$ being minimal.
 
 \end{proof}


\section{Generalized Dyck path triangulations}\label{sec:generalization}

In this section we show how Dyck path triangulations, and their extended
versions, have a natural generalization to triangulations $\ratdyckptri{rn}{n}$
of $\productp{rn-1}{n-1}$ for any positive integer $r$. This shows an interesting connection to rational Catalan combinatorics, which is an active area of recent interest, see for example~\cite{armstrong_results_2013,ArmstrongRhoadesWilliams2013}.

\begin{figure}[htbp]
\centering
  \begin{subfigure}{\textwidth} \centering
  \includegraphics[width=.6\textwidth]{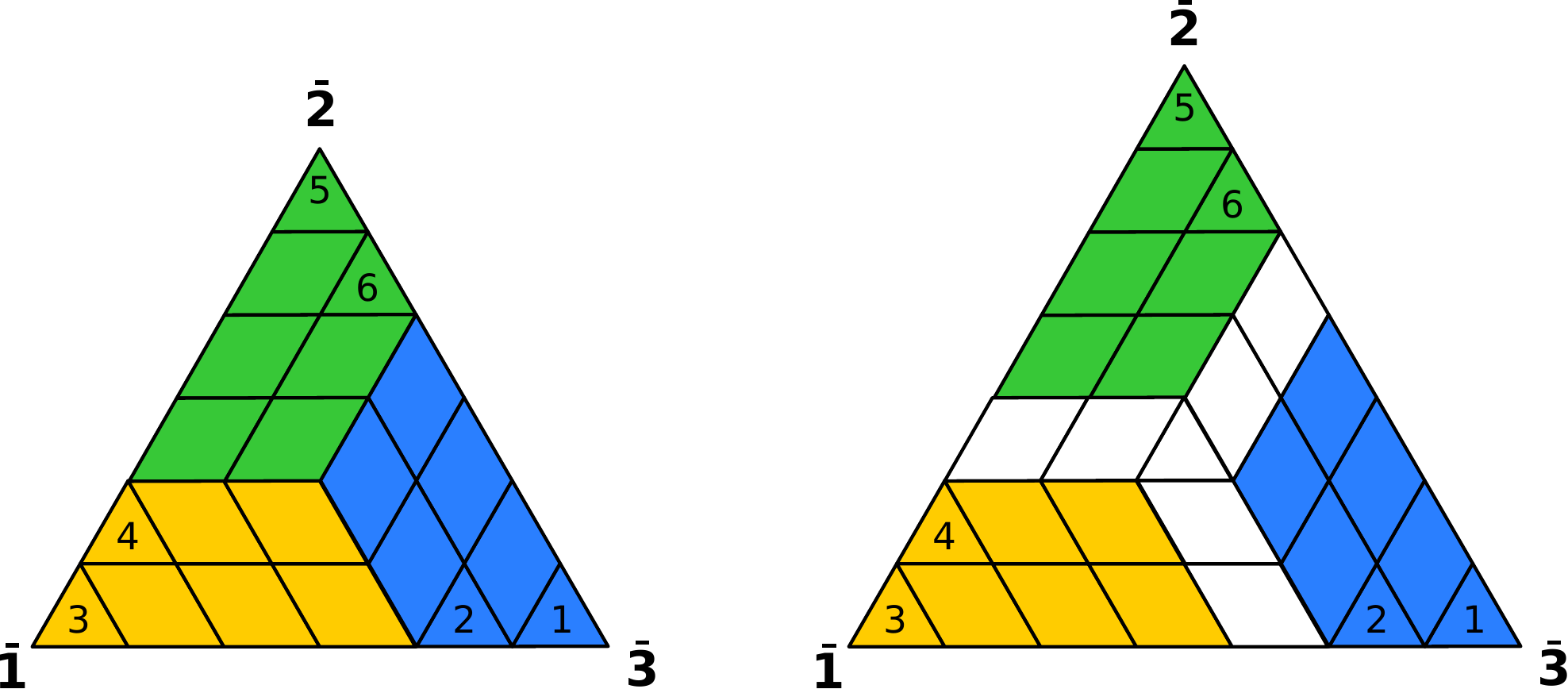}  
  \caption{Mixed subdivision representation of $\ratdyckptri{6}{3}$ and 
  $\ratdyckptriext{6}{3}$.}
  \label{fig:rationaldyckexamplesums} 
  \end{subfigure} 

  \begin{subfigure}{\textwidth} \centering
  \includegraphics[width=.9\textwidth]{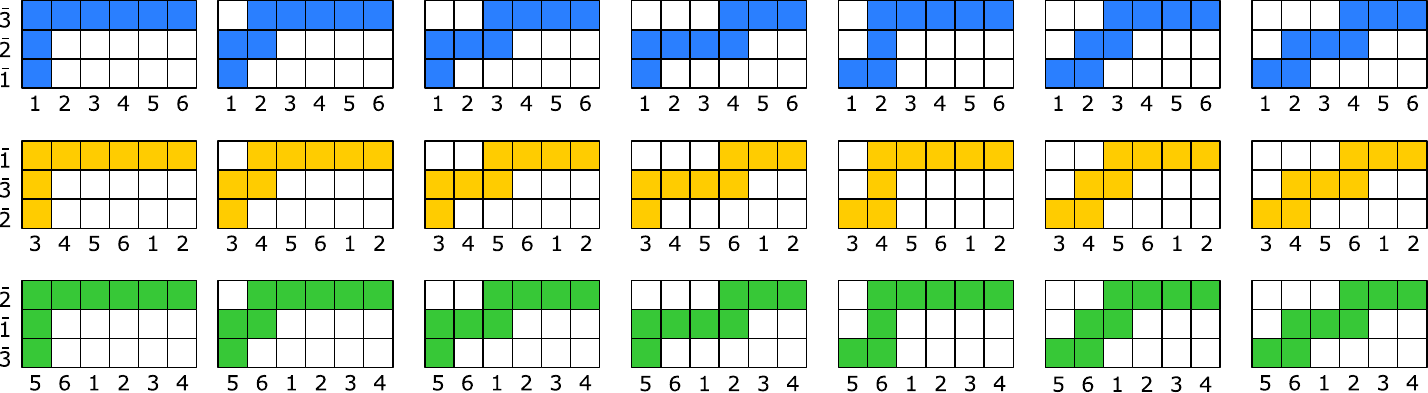}  
  \caption{Grid representation of $\ratdyckptri{6}{3}$.}
  \label{fig:rationaldyckexamplegrids} 
  \end{subfigure} 
  \caption{The triangulation $\ratdyckptri{6}{3}$ and its extension
$\ratdyckptriext{6}{3}$.}
  \label{fig:rationaldyckexample} 
\end{figure}

The Dyck path triangulation $\ratdyckptri{n}{n}=\dyckptri{n}$ exploits the
identity 
\[n\cdot \catnum{n-1}=\binom{2n-2}{n-1},\]
where $\catnum{n-1}$ is the $(n-1)$th Catalan number. Indeed, there are
$\catnum{n-1}$ Dyck paths from $(1,\overline 1)$ to $(n,\overline n)$, each of which represents a
simplex of (normalized) volume $1$ in $\productn$, and every such a simplex generates an orbit of $n$ simplices. Thus, altogether the orbits yield the correct (normalized) volume of $\productn$, equal to $\binom{2n-2}{n-1}$.

The triangulation $\ratdyckptri{rn}{n}$ of $\productp{rn-1}{n-1}$ analogously exploits the identity 
\[n\cdot \ratcatnum{n}{rn-1}=\binom{(r+1)n-2}{n-1},\]
where $\ratcatnum{a}{b}=\frac{1}{a+b}\binom{a+b}{a}$, for $a$ and $b$ relatively prime, are known as the \defn{rational Catalan
numbers}.

Define a \defn{$({rn,n})$-Dyck path}\footnote{The standard definition of a
rational $(a,b)$-Dyck path is
slightly different: it uses a grid from $(0,0)$ to $(a,b)$ and imposes that
$i<rj$ for any $i\neq 0,b$. It is used, for example,
in~\cite{ArmstrongRhoadesWilliams2013}.} in the grid $\gr{rn}{n}$ as a monotonically increasing path from~$(1,\overline 1)$ to~$(rn,\overline n)$ such that every step 
$(i,\overline j)$ satisfies~$i\leq rj$. There are exactly $\ratcatnum{n}{rn-1}$ such
paths. The \defn{$({rn,n})$-Dyck path triangulation} $\ratdyckptri{rn}{n}$ is the
triangulation of $\productp{rn-1}{n-1}$ that has as maximal simplices the
$({rn,n})$-Dyck paths together with their orbit under the action that \linebreak maps
$(i,\overline j)\mapsto (i+r \pmod{rn},\overline j+1\pmod n)$. We show an example in
Figure~\ref{fig:rationaldyckexample}.

\begin{theorem}
 The $({rn,n})$-Dyck path triangulation $\ratdyckptri{rn}{n}$ is a triangulation
of $\productp{rn-1}{n-1}$. Moreover, it is regular.
\end{theorem}

This theorem is immediate corollary of the following observation.

\begin{lemma}
 The restriction of $\dyckptri{rn}$ to the facet of $\productp{rn-1}{rn-1}$
spanned by the vertices $(\bfe_i,\bfe_{rj})$ with $i\in [rn]$, $j\in [n]$
coincides with $\ratdyckptri{rn}{n}$.
\end{lemma}

Since the proof is straightforward, we provide an illustrative example: in Figure~\ref{fig:rationaldyck} we obtain $\ratdyckptri{4}{2}$
from $\dyckptri{4}$ (which was depicted in Figure~\ref{fig:dyck4}).
In general, to recover $\ratdyckptri{rn}{n}$ from $\dyckptri{rn}$, we just
need to remove the rows of the grid $\gr{rn}{rn}$ that are not labeled by
multiples of~$r$; then we relabel the rows by $\overline j\mapsto\overline 
j/r$. Observe that not all simplices in the orbit of a $({rn,rn})$-Dyck path in
$\gr{rn}{rn}$ give simplices of $\productp{rn-1}{n-1}$ of maximal dimension, but
only those obtained by a shift divisible by $r$.            

\begin{figure}[htbp]
\centering
\includegraphics[width=.8\textwidth]{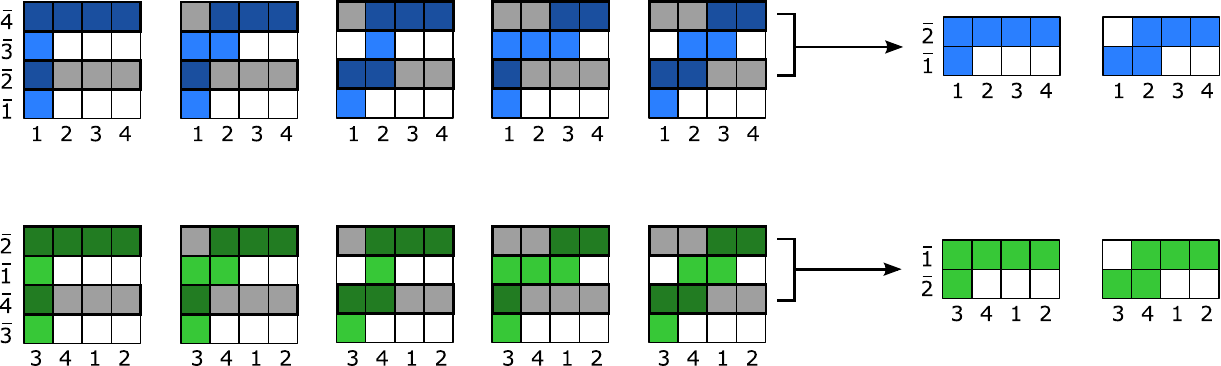}  
\caption{Obtaining $\ratdyckptri{4}{2}$ from $\dyckptri{4}$.}
\label{fig:rationaldyck} 
\end{figure}                    

The arguments from Section~\ref{sec:DyckTriangulation} leading to the non-extendable triangulation in Theorem~\ref{theorem:non-extend} can be reproduced in the setting of the $({rn,n})$-Dyck path triangulation:
\begin{itemize}
\item Restricting the extended Dyck path triangulation
$\dyckptriext{rn}$ to the facet of $\productp{rn}{rn-1}$
spanned by the vertices $(\bfe_i,\bfe_{rj})$ with $i\in [rn+1]$, $j\in [n]$ defines the \defn{extended $({rn,n})$-Dyck path triangulation}
$\ratdyckptriext{rn}{n}$. 
\item $\ratdyckptriext{rn}{n}$ bears the Dyck path triangulation $\dyckptri{n}$ on the face $\bfF$ of $\productp{rn-1}{n-1}$ spanned by $\{(\bfe_{ri},\bfe_j)\colon i\in[n], j\in[n]\}$.
\item In the restriction of $\ratdyckptriext{rn}{n}$ to
$\semiskelp{rn-1}{n-1}{n-1}$, we can flip the triangulation of $\bfF$ to
$\dyckptriflip{n}$, without altering the triangulations of the remaining faces
of $\semiskelp{rn-1}{n-1}{n-1}$. In particular, this supplies further examples
of non-extendable partial triangulations of~$\semiskelp{rn-1}{n-1}{n-1}$.
\end{itemize}

We end the article with an open question. If $a,b\in \NN$ are coprime, then 
\[a\cdot \ratcatnum{a}{b}=\binom{a+b-1}{a-1} \hspace{1cm} \text{or} \hspace{1cm} (a+b)\cdot \ratcatnum{a}{b}=\binom{a+b}{a}\]
\begin{question}
Is there a Dyck path triangulation of $\productp{a-1}{b}$ by rational Dyck paths that captures the first identity? or, is there a Dyck path triangulation of $\productp{a}{b}$ by rational Dyck paths that captures the second identity?
\end{question}


\bibliographystyle{amsplain}
\bibliography{dyckptri}

\providecommand{\bysame}{\leavevmode\hbox to3em{\hrulefill}\thinspace}
\providecommand{\MR}{\relax\ifhmode\unskip\space\fi MR }
\providecommand{\MRhref}[2]{%
  \href{http://www.ams.org/mathscinet-getitem?mr=#1}{#2}
}
\providecommand{\href}[2]{#2}
\begin{thebibliography}{10}

\bibitem{ArdilaBilley2007}
Federico Ardila and Sara Billey, \emph{{Flag arrangements and triangulations of
  products of simplices}}, Adv. Math. \textbf{214} (2007), no.~2, 495--524.

\bibitem{ArdilaCeballos2013}
Federico Ardila and Cesar Ceballos, \emph{{Acyclic systems of permutations and
  fine mixed subdivisions of simplices}}, Discrete Comput. Geom. \textbf{49}
  (2013), no.~3, 485--510.

\bibitem{ArdilaDevelin2009}
Federico Ardila and Mike Develin, \emph{{Tropical hyperplane arrangements and
  oriented matroids}}, Math. Z. \textbf{262} (2009), no.~4, 795--816.

\bibitem{armstrong_results_2013}
Drew Armstrong, Christopher R.~H. Hanusa, and Brant~C. Jones, \emph{Results and
  conjectures on simultaneous core partitions},
  \href{http://arxiv.org/abs/1308.0572}{arXiv:1308.0572v1} (2013), 17 pages.

\bibitem{ArmstrongRhoadesWilliams2013}
Drew Armstrong, Brendon Rhoades, and Nathan Williams, \emph{{Rational
  associahedra and noncrossing partitions}}, Electron. J. Combin. \textbf{20}
  (2013), no.~3, Paper 54, 27.

\bibitem{BabsonBillera}
Eric~K. Babson and Louis~J. Billera, \emph{{The geometry of products of
  minors}}, Discrete Comput. Geom. \textbf{20} (1998), no.~2, 231--249.

\bibitem{OrientedMatroids1993}
Anders Bj{\"o}rner, Michel {Las Vergnas}, Bernd Sturmfels, Neil White, and
  G{\"u}nter~M. Ziegler, \emph{{Oriented matroids.}}, {Encyclopedia of
  Mathematics and Its Applications. 46. Cambridge: Cambridge University Press.
  516 p. }, 1993.

\bibitem{ConcaHostenThomas2005}
Aldo Conca, Serkan Hosten, and Rekha~R. Thomas, \emph{Nice initial complexes of
  some classical ideals}, Algebraic and geometric combinatorics, Contemp. Math.
  \textbf{423} (2007), 11--42.

\bibitem{DeLoera1996}
J.~A. de~Loera, \emph{{Nonregular triangulations of products of simplices}},
  Discrete Comput. Geom. \textbf{15} (1996), no.~3, 253--264.

\bibitem{DeLoeraRambauSantos}
Jes{\'u}s~A. {De Loera}, J{\"o}rg Rambau, and Francisco Santos,
  \emph{{Triangulations}}, {Algorithms and Computation in Mathematics},
  vol.~25, Springer-Verlag, Berlin, 2010, Structures for algorithms and
  applications.

\bibitem{SturmfelsDevelin2004}
Mike Develin and Bernd Sturmfels, \emph{{Tropical convexity}}, Doc. Math.
  \textbf{9} (2004), 1--27 (electronic).

\bibitem{Dey1993}
Tamal~Krishna Dey, \emph{{On counting triangulations in $d$ dimensions.}},
  Computational Geometry: Theory and Applications \textbf{3} (1993), no.~6,
  315--325.

\bibitem{GKZ}
I.~M. Gelfand, M.~M. Kapranov, and A.~V. Zelevinsky, \emph{{Discriminants,
  resultants, and multidimensional determinants}}, {Mathematics: Theory \&
  Applications}, Birkh{\"a}user Boston Inc., Boston, MA, 1994.

\bibitem{Haiman1991}
Mark Haiman, \emph{{A simple and relatively efficient triangulation of the
  {$n$}-cube}}, Discrete Comput. Geom. \textbf{6} (1991), no.~4, 287--289.

\bibitem{HillarSullivant12}
Christopher~J Hillar and Seth Sullivant, \emph{{Finite Gr{\"o}bner bases in
  infinite dimensional polynomial rings and applications}}, Adv. Math.
  \textbf{229} (2012), no.~1, 1--25.

\bibitem{HostenSullivant07}
Serkan Ho\c{s}ten and Seth Sullivant, \emph{A finiteness theorem for {M}arkov
  bases of hierarchical models}, J. Combin. Theory Ser. A \textbf{114} (2007),
  no.~2, 311--321.

\bibitem{Horn2012}
Silke Horn, \emph{{A topological representation theorem for tropical oriented
  matroids}}, {24th {I}nternational {C}onference on {F}ormal {P}ower {S}eries
  and {A}lgebraic {C}ombinatorics ({FPSAC} 2012)}, {Discrete Math. Theor.
  Comput. Sci. Proc., AR}, Assoc. Discrete Math. Theor. Comput. Sci., Nancy,
  2012, pp.~135--146.

\bibitem{OhYoo2011}
Suho Oh and Hwanchul Yoo, \emph{{Triangulations of
  {$\Delta_{n-1}\times\Delta_{d-1}$} and tropical oriented matroids}}, {23rd
  {I}nternational {C}onference on {F}ormal {P}ower {S}eries and {A}lgebraic
  {C}ombinatorics ({FPSAC} 2011)}, {Discrete Math. Theor. Comput. Sci. Proc.,
  AO}, Assoc. Discrete Math. Theor. Comput. Sci., Nancy, 2011, pp.~717--728.

\bibitem{suho_triangulations_2013}
\bysame, \emph{{Triangulations of $\Delta_{n-1} \times \Delta_{d-1}$ and
  Matching Ensembles}},
  \href{http://arxiv.org/abs/1311.6772}{arXiv:1311.6772v1} (2013), 13 pages.

\bibitem{OrdenSantos2003}
David Orden and Francisco Santos, \emph{{Asymptotically efficient
  triangulations of the {$d$}-cube}}, Discrete Comput. Geom. \textbf{30}
  (2003), no.~4, 509--528.

\bibitem{Santos2000}
Francisco Santos, \emph{{A point set whose space of triangulations is
  disconnected}}, J. Amer. Math. Soc. \textbf{13} (2000), no.~3, 611--637.

\bibitem{Santos2005}
\bysame, \emph{{The {C}ayley trick and triangulations of products of
  simplices}}, Integer Points in Polyhedra -- Geometry, Number Theory, Algebra,
  Optimization, Contemp. Math. \textbf{374} (2005), 151--177.

\bibitem{Santos2013}
\bysame, \emph{{Some acyclic systems of permutations are not realizable by
  triangulations of a product of simplices}}, Algebraic and Combinatorial
  Aspects of Tropical Geometry, Contemp. Math. \textbf{589} (2013), 317--328.

\bibitem{Snowden13}
Andrew Snowden, \emph{{Syzygies of {S}egre embeddings and {$\Delta
  $}-modules}}, Duke Math. J. \textbf{162} (2013), no.~2, 225--277.

\bibitem{Sturmfels1996}
Bernd Sturmfels, \emph{{Gr{\"o}bner bases and convex polytopes}}, {University
  Lecture Series}, vol.~8, American Mathematical Society, Providence, RI, 1996.

\end{thebibliography}
\end{document}